\documentclass[a4paper,12pt]{amsart}

\usepackage[foot]{amsaddr}
\usepackage[T1]{fontenc}
\usepackage[utf8]{inputenc}
\usepackage{palatino}
\usepackage{amsmath, amssymb,mathtools}
\usepackage[usenames,dvipsnames]{xcolor}
\usepackage{graphicx}
\usepackage{subfigure}
\usepackage{minitoc}
\usepackage{tikz}
\usepackage{enumerate}
\usepackage[margin=0.96 in]{geometry}
\usepackage{bbm}
\usepackage[numbers,sort&compress]{natbib}
\usepackage[colorlinks=true]{hyperref}
\hypersetup{urlcolor=blue, citecolor=red}
\usepackage{crossreftools}

\DeclarePairedDelimiter{\abs}{\lvert}{\rvert}
\DeclarePairedDelimiter{\norm}{\lVert}{\rVert}

\DeclarePairedDelimiter{\skp}{\langle}{\rangle}

\DeclareMathAlphabet{\mathup}{OT1}{\familydefault}{m}{n}
\newcommand{\dx}[1]{\mathop{}\!\mathup{d} #1}

\DeclarePairedDelimiter{\prt}{(}{)}
\DeclarePairedDelimiter{\brk}{[}{]}

\newcommand{\N}{{\mathbb N}}
\newcommand{\R}{{\mathbb R}}
\newcommand{\Rd}{{\mathbb R^d}}

\usepackage{amsthm}
\theoremstyle{plain}
\newtheorem{theorem}{Theorem}[section]
\newtheorem{lemma}[theorem]{Lemma}
\newtheorem{proposition}[theorem]{Proposition}
\newtheorem{corollary}[theorem]{Corollary}

\theoremstyle{remark}
\newtheorem{remark}[theorem]{\bf Remark}
\newtheorem{definition}[theorem]{\bf Definition}

\newcommand{\ds}{\displaystyle}

\newcommand{\partialt}[1]{\frac{\partial #1}{\partial t}}

\newcommand{\fpartial}[1]{\frac{\partial}{\partial #1}}

\newcommand{\ou}{\overline{u}_k}
\newcommand{\ov}{\overline{v}_k}
\newcommand{\ow}{\overline{w}_k}

\begin{document}

\title[Convergence of Numerical Methods]{Convergence of a discrete-in-time Approximation to a Degenerate Parabolic-Hyperbolic System}

\author{Julia Hauser$^{1}$}
\author{Hideki Murakawa$^{2}$}
\author{Markus Schmidtchen$^{3}$}

\address{$^{1}$ Joanneum Research Forschungsgesellschaft mbH, Steyrergasse 17, 8010, Graz, Austria.}
\address{$^{2}$ Faculty of Advanced Science and Technology, Ryukoku University, 1-5 Yokotani Seta Oe-cho Otsu Shiga 520-2194, Japan}
\address{$^{3}$ Institute of Scientific Computing, Faculty of Mathematics, TUD Dresden University of Technology, Germany.}

\maketitle
\begin{abstract}
	In this paper we consider an implicit semi-discrete approximation of a degenerate reaction-cross-diffusion system. Due to the symmetry in the parabolic part, this system is known to preserve segregation of densities -- initially non-overlapping densities belonging to different species remain segregated for all times, which leads to internal layers between different species. We show that time-discrete approximations exist and converge to a weak solution, as the timestep goes to zero.\\[0.5em]
\end{abstract}{}

\vskip .4cm
\begin{flushleft}
	\noindent{\makebox[1in]\hrulefill}
\end{flushleft}
\begin{flushright}
	\noindent{\makebox[1in]\hrulefill}
\end{flushright}
\vskip 1.5cm

\section{Introduction}
We propose a semi-discrete scheme for the degenerate parabolic cross-diffusion system for two population densities, $u,v \geq 0$:
\begin{align}
	\label{eq:cts-system}
	\left\{
	\begin{array}{rl}
		\displaystyle \frac{\partial u}{\partial t} - \nabla \cdot\prt*{u \nabla w^\gamma} \!\!\!  & = uF_u(w)+vG_u(w), \\[1em]
		\displaystyle \frac{\partial v}{\partial t}  - \nabla \cdot\prt*{v \nabla w^\gamma} \!\!\! & = uF_v(w)+vG_v(w),
	\end{array}
	\right.
\end{align}
where  $w := u + v$ denotes the total population and $\gamma \ge 1$ is a constant. The system is posed on $\Omega_T:=\Omega \times (0,T)$, where $\Omega \subset \Rd$ is a bounded domain with smooth boundary $\partial \Omega$, and $T>0$. We impose no-flux boundary conditions on $\partial \Omega$. The system is equipped with nonnegative initial data $u_0,v_0 \in L^\infty(\Omega)$. Finally,  $F_u, G_u$ (resp. $F_v, G_v$) model the reaction phenomena of species $u$ (resp. $v$).

The system was first proposed in \cite{GP1984} as a model for populations that avoid overcrowding and in \cite{BT1983} for polymorphic populations in epidemiological contexts. Unlike linear diffusion models of Fisher-KPP type, cf. \cite{KPP1988,Fis1937}, the system is formally degenerate parabolic. Indeed, its one species counterpart reads
\begin{align*}
	\partialt w = \nabla \cdot(w \nabla w^\gamma) = \frac{\gamma}{\gamma + 1}\Delta w^{\gamma + 1},
\end{align*}
which is the well-known filtration equation or porous medium equation, see \cite{Vaz2007} and references therein. This equation can also be interpreted as a continuity equation where the velocity is related to the pressure gradient ($v = - \nabla p$) and the pressure is related to the density through the constitutive law $p = w^\gamma$.  A striking feature of the porous medium equation is its finite speed of propagation which is desirable from a biological point of view in contrast to linear diffusion which infinite speed of propagation. Furthermore, solutions to the porous medium equation lose smoothness at the boundary of their support unlike solutions to its linear counterpart.

Another sudden loss of regularity occurs due to the presence of other species. Indeed, initially segregated densities remain segregated for all times and, as their associated supports touch, a free internal layer forms \cite{Bertsch1987, Bertsch2012, Bertsch2020} across which the density exhibits jumps.

Despite a lot of analytical progress and a variety of new approaches to proving existence of solutions in recent years \cite{CFSS2018, GPS2019, BPPS2020, PX2020, LX2021, DHJ2023, DDMS2024, Jac2023, Jac2025}, there has only been little progress on numerical approximations. Indeed, to the best of our knowledge, the only numerical approximation of \eqref{eq:cts-system} is due to \cite{hopf2025convergence} where a fully discrete structure-preserving finite-volume method is presented and measure-valued solutions are constructed. Indeed, their numerical solution converges to a Young-measure solution which is shown to concentrate whenever System \eqref{eq:cts-system} admits a strong solutions by a weak-strong uniqueness argument. While it is know that local strong solutions exist \cite{DHJ2023}, solutions are, in general, not smooth. This is due to the aforementioned non-mixing property of \eqref{eq:cts-system}, i.e., densities with jumps across the interface between supports, cf. \cite{GP1984}.

In this paper,
we analyse a time-discrete implicit Euler approximation of System \eqref{eq:cts-system}.  Let $\tau = T/N$, with $N\in \mathbb{N}$, be the timestep size, and use the same initial data $u_0, v_0$ as in System~\eqref{eq:cts-system}. For $k=1,2,\dots, N$, we consider
\begin{align}
	\label{eq:intro-scheme}
	\begin{split}
		\left\{
		\begin{array}{lll}
			\dfrac1\tau(u_k - u_{k-1}) - \nabla \cdot\prt*{u_k \nabla w_k^\gamma} \!\!\!    & = R_u(u_k,v_k,w_k), \\[1em]
			\dfrac1\tau(v_k \, - v_{k-1})  - \nabla \cdot\prt*{v_k \nabla w_k^\gamma}\!\!\! & = R_v(u_k,v_k,w_k), \\
		\end{array}
		{\quad \mbox{in } \Omega},
		\right.
	\end{split}
\end{align}
with $w_k:=u_k+v_k$ and
equipped with zero Neumann conditions for both species.
In this paper, we show the existence of a weak solution of this scheme and its convergence to a weak solution of System \eqref{eq:cts-system}, as $\tau$ goes to zero.

We emphasize that this is not yet a fully discrete numerical scheme, since the spatial variable is not discretized.
A fully discrete scheme can be obtained by combining the present time discretization with a suitable spatial discretization.
There are several possible choices for the spatial discretization, such as finite difference methods, finite element methods, or finite volume methods, depending on the structure of the problem and the intended application.
The analysis of the time-discrete scheme therefore plays an important role as a fundamental step toward the construction and analysis of fully discrete numerical schemes.

The rest of this paper is organised as follows.
In the next section, we state notations and definitions, and give our main results.
In Section \ref{sec:existence-discrete-solns}, propose a regularised scheme, and show that it can be solved. Section \ref{sec:apriori} is dedicated to deriving estimates that are uniform in the regularisation parameter $\epsilon>0$ and the timestep $\tau >0$. Finally, in Section \ref{sec:compconv}, we show that we can remove the regularisation parameter and that the resulting time-discrete implicit Euler approximation converges to a weak solution of System \eqref{eq:cts-system}.

\section{Assumptions and main results.}
Throughout this paper, we impose the following assumptions.
\begin{enumerate}
	\renewcommand{\labelenumi}{(H\theenumi)}
	\item \label{H1} $\gamma \ge 1$.
	\item \label{H2} $u_0, v_0\in L^\infty(\Omega)$ and $w_0:=u_0+v_0$ satisfy $0\le u_0,v_0,w_0\le M$ for a positive constant $M$ and $w_0 \in  H^1(\Omega)$.
	\item \label{H3} $F_u,F_v,G_u$ and $G_v$ are continuous and satisfy
	      \begin{align*}
		      F(w) & := F_u(w)+F_v(w)\leq 0, \quad  \text{ for } w \geq  M, \\
		      G(w) & := G_u(w)+G_v(w)\leq 0, \quad  \text{ for } w \geq  M,
	      \end{align*}
	      and $\min(F_v(w),G_u(w))\geq 0$ for $0\leq w\leq M$.
	      We set
	      \begin{align}
		      \label{eq:def-RuRv}
		      R_u(u,v,w) := uF_u(w)+vG_u(w), \quad \text{and} \quad
		      R_v(u,v,w) := uF_v(w)+vG_v(w).
	      \end{align}
	      In addition, we assume
	      \begin{align}
		      \label{eq:RuRv_assumptions}
		      R_u(u,v,w) \geq 0, \quad \text{and} \quad R_v(u,v,w) \geq 0,
	      \end{align}
	      for all $0\leq u,v,w \leq M$. \\
\end{enumerate}

\noindent
System~\eqref{eq:cts-system} is understood in the following weak sense.
\begin{definition}[Weak solution of System~\eqref{eq:cts-system}]
	\label{def:weak-solution}
	A pair $(u,v)$ is called weak solution of System \eqref{eq:cts-system} if {$u, v \in L^\infty(\Omega_T)\cap H^1(0,T;H^1(\Omega)')$ with $w=u+v$, $w^\gamma \in L^2(0,T;H^1(\Omega))$}  such that
	\begin{subequations}
		\begin{align*}
			\int_0^T {\left\langle \frac{\partial u}{\partial t}, \varphi\right\rangle_{H^1(\Omega)', H^1(\Omega)} \dx t} + \int_0^T\int_\Omega u \nabla w^\gamma \cdot \nabla \varphi \dx x \dx t & = \int_0^T\int_\Omega R_u(u, v, w) \varphi\dx x \dx t, \\
			\int_0^T {\left\langle \frac{\partial v}{\partial t} , \psi\right\rangle_{H^1(\Omega)', H^1(\Omega)} \dx t} + \int_0^T\int_\Omega v \nabla w^\gamma \cdot \nabla \psi \dx x \dx t      & = \int_0^T\int_\Omega R_v(u, v, w) \psi\dx x \dx t,
		\end{align*}
	\end{subequations}
	for all $\varphi, \psi \in L^2(0,T;H^1(\Omega))$, and $u(0)=u_0$ as well as $v(0) = v_0$.
\end{definition}

We also consider the weak formulation of the time-discrete approximation, System \eqref{eq:intro-scheme}, as follows.

\begin{definition}[Weak solution of Scheme~\eqref{eq:intro-scheme}]
	Given initial data $u_0,v_0$, a sequence
	$
		(u_k,v_k)_{k=0}^N
	$
	is called a weak solution of Scheme~\eqref{eq:intro-scheme} if
	$u_k,v_k\in L^\infty(\Omega)$ with
	$w_k:=u_k+v_k$,
	$w_k^\gamma\in H^{1}(\Omega)$ such that
	\[
		\frac{1}{\tau} \int_\Omega (u_k-u_{k-1})\,\varphi\,\dx x
		+
		\int_\Omega u_k\nabla w_k^\gamma\cdot\nabla\varphi\,\dx x
		=
		\int_\Omega R_u(u_k,v_k,w_k)\,\varphi\,\dx x
	\]
	\[
		\frac{1}{\tau} \int_\Omega (v_k-v_{k-1})\,\psi\,\dx x
		+
		\int_\Omega v_k\nabla w_k^\gamma\cdot\nabla\psi\,\dx x
		=
		\int_\Omega R_v(u_k,v_k,w_k)\,\psi\,\dx x,
	\]
	for all $\varphi,\psi\in H^{1}(\Omega)$ and for every $k=1,2,\dots,N$.
\end{definition}

Subsequently, we will make use of the following interpolations of sequences of the approximations we construct.
\begin{definition}[Piecewise constant / linear interpolation]
	\label{def:pw-interpolations}
	For a sequence of functions $(f_k)_{k=0}^N$, with $f_k:\Omega \to \R$, for any $k=1, \ldots, N$,  we define
	\begin{align}
		\overline f_{\tau}(x,t) := f_{k}(x),
	\end{align}
	whenever $t\in (t^{k-1}, t^{k}]$ for $k=1,\ldots, N$, with $ \overline f_{\tau}(x,0) := f_{0}(x)$. We call $\overline f_\tau$ the piecewise constant interpolation of $(f_k)_{k=0}^N$. Moreover, we set
	\begin{align}
		\widehat f_{\tau}(x,t) := f_{k-1}(x) + \frac{t-t^{k-1}}{\tau} \big(f_{k}(x) - f_{k-1}(x)\big),
	\end{align}
	whenever $t\in (t^{k-1}, t^{k}]$, for $k=1,\ldots, N$, with $ \widehat f_{\tau}(x,0) := f_{0}(x)$, and call $\widehat f_\tau$ the piecewise linear interpolation of $(f_k)_{k=0}^N$.
\end{definition}

The main results of this paper is as follows.

\begin{theorem}[Existence of the weak solution of Scheme~\eqref{eq:intro-scheme}]
	\label{thm:existence-scheme}
	Let $\mathrm{(H1)}$--$\mathrm{(H3)}$ be satisfied. Then, there exists a weak solution of the scheme~\eqref{eq:intro-scheme}.
\end{theorem}

\begin{theorem}[Convergence of the scheme]
	\label{thm:convergence}
	Let $\mathrm{(H1)}$--$\mathrm{(H3)}$ hold and let $(u_k,v_k)_{k=0}^N$ be the weak solution of Scheme~\eqref{eq:intro-scheme} and set $(w_k)_{k=0}^N=(u_k+v_k)_{k=0}^N$. Then, there exists subsequences of the associated interpolations (not relabeled) and a weak solution $(u,v)$ of System~\eqref{eq:cts-system} with $w=u+v$ such that
	\[
		\begin{aligned}
			(\overline{u}_{\tau},\overline{v}_{\tau},\overline{w}_{\tau})
			 & \overset{*}{\rightharpoonup} (u,v,w)
			 &                                      & \text{weakly-* in } L^\infty(\Omega_T),                                                  \\
			(\widehat{u}_{\tau},\widehat{v}_{\tau},\widehat{w}_{\tau})
			 & \rightharpoonup (u,v,w)
			 &                                      & \text{weakly in } H^1(0,T;H^1(\Omega)'),                                                 \\
			 &                                      &                                            & \text{and weakly-* in } L^\infty(\Omega_T), \\[0.4em]
			\overline{w}_{\tau},\widehat{w}_{\tau}
			 & \to w
			 &                                      & \text{strongly in } L^q(\Omega_T)
			\quad (1\le q<+\infty),                                                                                                            \\
			\overline{w}_{\tau}^{\,p},\widehat{w_{\tau}^{p}}
			 & \rightharpoonup w^p
			 &                                      & \text{weakly in } L^2(0,T;H^1(\Omega))
			\quad \left(\frac{\gamma+1}{2}\le p<+\infty\right),                                                                                \\
			\widehat{w_{\tau}^{\gamma+1}}
			 & \to w^{\gamma+1}
			 &                                      & \text{strongly in } C([0,T];L^q(\Omega))
			\quad (1\le q<+\infty),                                                                                                            \\
			\overline{w}_{\tau}^{\gamma+1},
			\widehat{w_{\tau}^{\gamma+1}}
			 & \to w^{\gamma+1}
			 &                                      & \text{strongly in }  L^2(0,T;H^1(\Omega)),
			\text{ and weakly-* in } L^\infty(0,T;H^1(\Omega)),
		\end{aligned}
	\]
	as $\tau \to 0$.
\end{theorem}

\section{Existence of Solutions to the Scheme}
\label{sec:existence-discrete-solns}
In order to prove the existence of solutions to Scheme~\eqref{eq:intro-scheme}, let us consider an elliptic regularisation first. Specifically, we consider the following 3-component system
\begin{subequations}
	\label{eq:scheme}
	\begin{align}
		\dfrac1\tau(w_k - w_{k-1}) - \epsilon \Delta w_k - \dfrac{\gamma}{\gamma +1} \Delta w_k^{\gamma + 1}
		 & = (R_{u}+R_v )(u_k,v_k,w_k),\label{eq:scheme:w} \\[0.5em]
		\dfrac1\tau(u_k - u_{k-1}) - \epsilon \Delta u_k - \nabla \cdot\prt*{u_k \nabla w_k^\gamma}
		 & = R_u(u_k,v_k,w_k),\label{eq:scheme:u}          \\[0.5em]
		\dfrac1\tau(v_k - v_{k-1}) - \epsilon \Delta v_k - \nabla \cdot\prt*{v_k \nabla w_k^\gamma}
		 & = R_v(u_k,v_k,w_k),\label{eq:scheme:v}
	\end{align}
\end{subequations}
in $\Omega$ equipped with zero Neumann data on the boundary $\partial \Omega$. This problem is also  understood in the weak sense. We can show that this regularized scheme admits a solution.

\begin{lemma}
	\label{lem:buffer-this-didnt-have-a-label}
	Let $\mathrm{(H1)}$--$\mathrm{(H3)}$ be satisfied. Then, for any $\epsilon>0$ and  $\tau>0$, there exists a weak solution $(u_{k,\epsilon}, v_{k,\epsilon}, w_{k,\epsilon})\in [H^1(\Omega)\cap L^\infty(\Omega)]^3$ of System \eqref{eq:scheme} such that $0\le u_{k,\epsilon}, v_{k,\epsilon}, w_{k,\epsilon}\le M$ and $w_{k,\epsilon}=u_{k,\epsilon}+v_{k,\epsilon}$.
\end{lemma}

In this section, we consider the weak solution $(u_{k,\epsilon}, v_{k,\epsilon}, w_{k,\epsilon})$ of Scheme~\eqref{eq:scheme}, yet, we shall omit the subscript $\epsilon$ and simply write $(u_k, v_k, w_k)$ for ease of notation.

Before we begin the proof, let us introduce the cut-off function, $\phi$, defined via
\begin{align*}
	\phi(r) = \max(0, \min(r,M)) =
	\begin{cases}
		0, & \text{ if } r \leq 0,  \\
		r, & \text{ if } 0 < r <M,  \\
		M, & \text{ if } r \geq  M,
	\end{cases}
\end{align*}
for $r \in \R$.
\begin{proof}
	We proceed by induction. For some $k \in \N$, $k\geq 1$, let $(u_{k-1},v_{k-1})\in L^\infty(\Omega)\times L^\infty(\Omega)$ with $w_{k-1} \in H^1(\Omega) \cap L^\infty(\Omega)$ such that $0 \leq u_{k-1}, v_{k-1}, w_{k-1} \leq M$.

	The construction of $(u_k, v_k, w_k) \in [H^1(\Omega)\cap L^\infty(\Omega)]^3$ is done in three steps. First, we show existence of solutions to a system in which the nonlinearity is frozen and by employing the Lax-Milgram theorem. Next, using a fixed-point argument, we show that there is a solution to the nonlinear problem. In the end, we show positivity and $L^\infty(\Omega)$ boundedness, which allows us to remove the cut-off.

	\textbf{Step 1. -- Existence of solutions to the linearised scheme}\\
	Let $(a,b,c) \in L^2(\Omega)^3$ be given and fixed and consider
	\begin{subequations}
		\label{eq:LinScheme}
		\begin{align}
			\label{eqn:pdew}
			\frac{ w_k -w_{k-1}}{\tau}- \epsilon \Delta w_k - \gamma \nabla \cdot ((\phi(a)+\phi(b))^{\gamma} \nabla w_k) & = (R_{u}+R_v )(\phi(a),\phi(b),\phi(c)),
		\end{align}
		\begin{align}
			\label{eqn:pdeu}
			\frac{ u_k -u_{k-1}}{\tau}- \epsilon \Delta u_k & = \gamma \nabla \cdot (\phi(a)(\phi(a)+\phi(b))^{\gamma-1} \nabla w_k) +  R_{u}(\phi(a),\phi(b),\phi(c)),
		\end{align}
		\begin{align}
			\label{eqn:pdev}
			\frac{ v_k -v_{k-1}}{\tau}- \epsilon \Delta v_k & = \gamma \nabla \cdot (\phi(b)(\phi(a)+\phi(b))^{\gamma-1} \nabla w_k)  +  R_v (\phi(a),\phi(b),\phi(c)),
		\end{align}
	\end{subequations}
	posed in $\Omega$ and equipped with zero Neumann data on $\partial \Omega$, each. In order to show solvability of Eq. \eqref{eqn:pdew} we will use the Lax-Milgram theorem. Indeed, Eq. \eqref{eqn:pdew} is to be understood in a weak sense, i.e., there has to hold
	\begin{align}
		\label{eqn:weakformw}
		\begin{split}
			\frac{1}{\tau}\int_\Omega w_k \varphi \dx x
			 & + \epsilon\int_\Omega \nabla w_k \cdot \nabla \varphi \dx x
			+ \gamma\int_\Omega (\phi(a)+\phi(b))^{\gamma}
			\nabla w_k \cdot \nabla \varphi \dx x                          \\
			 & =
			\int_\Omega
			\brk*{
				\frac{w_{k-1}}{\tau}
				+ \phi(a)F(\phi(c))
				+ \phi(b)G(\phi(c))
			}
			\varphi \dx x,
		\end{split}
	\end{align}
	for any $\varphi \in H^1(\Omega)$. For $f, g \in H^1(\Omega)$, let us define the bilinear form
	\begin{align*}
		\mathcal B_w [f, g] := \frac{1}{\tau}\int_\Omega f g \dx x	+ \epsilon \int_\Omega  \nabla f \cdot \nabla g \dx x + \gamma \int_\Omega (\phi(a)+\phi(b))^{\gamma}  \nabla f \cdot   \nabla g \dx x.
	\end{align*}
	Then $\mathcal B_w$ is bounded in $H^1(\Omega)$ since
	\begin{align*}
		\abs{\mathcal B_w [f,g]} \leq \frac{1}{\tau}\norm{f}_{L^2(\Omega)} \norm{g}_{L^2(\Omega)} + (\epsilon + \gamma(2M)^{\gamma})\norm{f}_{H^1(\Omega)}\norm{g}_{H^1(\Omega)},
	\end{align*}
	for $f, g\in H^1(\Omega)$. On the other hand, $\mathcal B_w$ is coercive in $H^1(\Omega)$ since $\phi(\cdot) \geq 0$ and therefore
	\begin{align*}
		\mathcal B_w [f,f] & \geq \frac{1}{\tau}\norm{f}^2_{L^2(\Omega)} + \epsilon \norm{\nabla f}_{L^2(\Omega)}^2 \\
		                   & \geq \min (\tau^{-1},\epsilon) \norm{f}^2_{H^1(\Omega)}.
	\end{align*}
	Thus, Eq. \eqref{eqn:pdew} can be written as
	\begin{align*}
		\mathcal B_w [w_k, \varphi] = \int_\Omega \brk*{\frac{w_{k-1}}{\tau} + \phi(a) F(\phi(c))+\phi(b) G(\phi(c))}\varphi \dx x,
	\end{align*}
	where the right-hand side is in $L^2(\Omega)$  since $F_u,F_v,G_u$ and $G_v$ are continuous and $(R_{u}+R_v )\circ \phi$ is bounded. Hence, we can apply the Lax-Milgram theorem and obtain a unique solution $w_k$ of Eq. \eqref{eqn:weakformw} such that
	\begin{align*}
		\norm{w_k}_{H^1(\Omega)}
		\leq
		\frac{1}{\min\{1,\tau\epsilon\}}
		\norm{
			w_{k-1}
			+\tau\phi(a)F(\phi(c))
			+\tau\phi(b)G(\phi(c))
		}_{L^2(\Omega)} .
	\end{align*}
	Now, let us turn our attention to the remaining two equations beginning with Eq. \eqref{eqn:pdeu}. We recall the variational formulation
	\begin{align*}
		\frac{1}{\tau} & \int_\Omega u_k \varphi \dx x + \epsilon \int_\Omega \nabla u_k \cdot \nabla \varphi \dx x                                                        \\
		               & = -\gamma\int_\Omega \phi(a)(\phi(a)+\phi(b))^{\gamma-1} \nabla w_k \cdot  \nabla\varphi \dx x  +\frac{1}{\tau} \int_\Omega u_{k-1} \varphi \dx x \\
		               & \quad +  \int_\Omega \big[\phi(a)F_u(\phi(c))+ \phi(b)G_u(\phi(c))\big]\varphi \dx x,
	\end{align*}
	for all $\varphi \in H^1(\Omega)$. Indeed, employing the Lax-Milgram theorem again, we shall show that there exists a unique solution, $u_k \in H^1(\Omega)$. Similar to before, we introduce
	\begin{align*}
		\mathcal B_u[f,g] & := \frac{1}{\tau} \int_\Omega f g \dx x + \epsilon \int_\Omega \nabla f\cdot \nabla g \dx x,
	\end{align*}
	for $f, g\in H^1(\Omega)$, and set
	\begin{align*}
		\mathcal R[g] :=
		 & - \gamma \int_\Omega \phi(a)(\phi(a)+\phi(b))^{\gamma-1} \nabla w_k \cdot \nabla g \dx x +\frac{1}{\tau} \int_\Omega u_{k-1} g \dx x \\
		 & + \int_\Omega \big(\phi(a)F_u(\phi(c))+ \phi(b)G_u(\phi(c))\big)  g\dx x,
	\end{align*}
	where we stress that the definition of $\mathcal R$ involves the unique solution, $w_k$, of Eq. \eqref{eqn:pdew}. Then, we rewrite the variational formulation in the concise form
	\begin{align*}
		\mathcal B_u[u_k, \varphi] = \mathcal R[\varphi].
	\end{align*}
	We readily verify that both
	\begin{align*}
		\abs*{\mathcal B_u[f, g]} & \leq  \left(\frac{1}{\tau}+\epsilon\right) \norm{f}_{H^1(\Omega)} \norm{g}_{H^1(\Omega)},
	\end{align*}
	and
	\begin{align*}
		\mathcal B_u[f,f] & \geq  \min \left(\tau^{-1},\epsilon\right) \norm{f}_{H^1(\Omega)}^2,
	\end{align*}
	as well as
	\begin{align*}
		\abs{\mathcal R [\varphi]} & \leq C\norm{\varphi}_{H^1(\Omega)},
	\end{align*}
	with
	$$
		C := \gamma M(2M)^{\gamma-1}\norm{\nabla w_k}_{L^2(\Omega)} + \frac{1}{\tau} \norm{u_{k-1}}_{L^2(\Omega)} +\|\phi(a)F_u(\phi(c))+ \phi(b)G_u(\phi(c))\|_{L^2(\Omega)}.
	$$
	Hence, the conditions of the Lax-Milgram theorem are met and we infer the existence of a unique $H^1(\Omega)$-solution, $u_k$, of Eq. \eqref{eqn:pdeu}. Then, the existence of a unique $v_k\in H^1(\Omega)$ solving Eq. \eqref{eqn:pdev} follows in the same vein.
	We set $z_k := u_k + v_k$, observe that $z_k$ and $w_k$ satisfy the same equation, and note that the difference $q_k:=z_k-w_k$ satisfies
	\[
		\frac{q_k}{\tau}-\epsilon\Delta q_k=0,
	\]
	in $H^1(\Omega)'$, together with homogeneous Neumann boundary conditions. Testing by $q_k$ gives
	\[
		\frac{1}{\tau}\norm{q_k}_{L^2(\Omega)}^2
		+
		\epsilon\norm{\nabla q_k}_{L^2(\Omega)}^2
		=0.
	\]
	Hence $q_k=0$, and therefore
	\[
		w_k=z_k=u_k+v_k.
	\]
	Consequently, for every fixed $(a,b,c)\in L^2(\Omega)^3$, there exists a unique
	triple
	\[
		(u_k,v_k,w_k)\in (H^1(\Omega))^3
	\]
	solving the linearised system \eqref{eq:LinScheme} in the weak sense. Moreover,
	\[
		w_k=u_k+v_k.
	\]

	\textbf{Step 2. -- Existence of a fixed point}\\
	In order to prove the  existence of solutions to the nonlinear problem, we shall use Sch\"afer's fixed-point theorem. To this end, let us define the mapping
	\begin{align*}
		\mathcal F:\big(L^2(\Omega)\big)^3 & \to \big(L^2(\Omega)\big)^3, \\
		(a,b,c)                            & \mapsto (u_k,v_k,w_k).
	\end{align*}
	Moreover, by the analysis in Step 1, we know that the image of $\mathcal F$ is contained in $(H^1(\Omega))^3$. In particular, by the bounds above, bounded sets in $(L^2(\Omega))^3$ are mapped into bounded sets  in $(H^1(\Omega))^3$, which are precompact in $(L^2(\Omega))^3$. We can use this property to show that $\mathcal F$ is continuous. To this end, let $(a_i, b_i, c_i) \to (a,b,c)$, as $i \to \infty$, strongly  in $(L^2(\Omega))^3$. Furthermore, let $(u_k^i, v_k^i, w_k^i) := \mathcal F(a_i, b_i, c_i)$ be the associated solutions. Since $\mathcal F$ maps bounded sets into precompact ones, there exists a subsequence, not relabeled, such that $(u_k^i, v_k^i, w_k^i) \to (\tilde u, \tilde v, \tilde w)$, strongly in $(L^2(\Omega))^3$ and weakly in $(H^1(\Omega))^3$, for some $\tilde u, \tilde v, \tilde w \in H^1(\Omega)$. We have to show that $(\tilde u, \tilde v, \tilde w) = (u_k, v_k, w_k) = \mathcal F(a,b,c)$.

	Using the strong compactness of the subsequence and the definition of $\phi$, we also have
	$$
		\prt{\phi(a_i), \phi(b_i), \phi(c_i)} \to \prt{\phi(a), \phi(b), \phi(c)}, \quad \mbox{as } i \to \infty,
	$$
	strongly in any $L^p$, with $p<\infty$, and the same is also true for the growth terms. Using the compactness, we can pass to the limit in System \eqref{eq:LinScheme}, i.e.,
	\begin{align*}
		\frac{ w_k^i -w_{k-1}}{\tau} & - \epsilon \Delta w_k^i - \gamma \nabla \cdot ((\phi(a_i)+\phi(b_i))^{\gamma} \nabla w_k^i)  = (R_{u}+R_v )(\phi(a_i),\phi(b_i),\phi(c_i)),     \\
		\frac{ u_k^i -u_{k-1}}{\tau} & - \epsilon \Delta u_k^i = \gamma \nabla \cdot (\phi(a_i)(\phi(a_i)+\phi(b_i))^{\gamma-1} \nabla w_k^i) +  R_{u}(\phi(a_i),\phi(b_i),\phi(c_i)), \\
		\frac{ v_k^i -v_{k-1}}{\tau} & - \epsilon \Delta v_k^i = \gamma \nabla \cdot (\phi(b_i)(\phi(a_i)+\phi(b_i))^{\gamma-1} \nabla w_k^i)  +  R_v (\phi(a_i),\phi(b_i),\phi(c_i)),
	\end{align*}
	understood in the variational sense, and obtain
	\begin{align*}
		\frac{ \tilde w - w_{k-1}}{\tau} & - \epsilon \Delta \tilde w - \gamma \nabla \cdot ((\phi(a)+\phi(b))^{\gamma} \nabla \tilde w)  = (R_{u}+R_v )(\phi(a),\phi(b),\phi(c)),   \\
		\frac{ \tilde u - u_{k-1}}{\tau} & - \epsilon \Delta \tilde u = \gamma \nabla \cdot (\phi(a)(\phi(a)+\phi(b))^{\gamma-1} \nabla \tilde w) +  R_{u}(\phi(a),\phi(b),\phi(c)), \\
		\frac{ \tilde v - v_{k-1}}{\tau} & - \epsilon \Delta \tilde v = \gamma \nabla \cdot (\phi(b)(\phi(a)+\phi(b))^{\gamma-1} \nabla \tilde w)  +  R_v (\phi(a),\phi(b),\phi(c)).
	\end{align*}
	By the uniqueness of the solution to the linearised problem with
	coefficients $(a,b,c)$, we identify
	\[
		(\tilde u,\tilde v,\tilde w)=\mathcal F(a,b,c).
	\]
	Since every convergent subsequence has the same limit, the whole sequence
	converges to $\mathcal F(a,b,c)$ in $[L^2(\Omega)]^3$, whence we infer the continuity of $\mathcal F$.

	Next, in order to apply the Sch\"afer fixed-point theorem, we need to prove that the set
	\[
		\mathcal S
		:=
		\left\{
		(\ou,\ov,\ow)\in [L^2(\Omega)]^3:
		(\ou,\ov,\ow)=\sigma\mathcal F(\ou,\ov,\ow)
		\text{ for some } \sigma\in(0,1)
		\right\}
	\]
	is bounded in $(L^2(\Omega))^3$.

	Let $(\ou,\ov,\ow)\in\mathcal S$, and let
	\[
		\mathcal F(\ou,\ov,\ow)
		=
		\frac1\sigma(\ou,\ov,\ow),
	\]
	for some $\sigma\in(0,1)$. Hence, in the variational formulation of the
	linearised problem, we take
	\[
		(a,b,c)=(\ou,\ov,\ow),
		\qquad
		\text{and} \qquad (u_k,v_k,w_k)=\frac1\sigma(\ou,\ov,\ow).
	\]
	Testing the equation for $\ow$ by $\ow$ yields
	\begin{align*}
		\frac{1}{\tau\sigma} \int_\Omega |\ow|^2 \dx x & + \frac{1}{\sigma} \int_\Omega (\epsilon +\gamma(\phi(\ou)+\phi(\ov))^{\gamma} ) |\nabla \ow|^2 \dx x                               \\
		                                               & = \frac{1}{\tau} \int_\Omega w_{k-1} \ow\dx x +  \int_\Omega \left[\phi(\ou)F(\phi(\ow ))+ \phi(\ov)G(\phi(\ow ))\right] \ow \dx x.
	\end{align*}
	From this equation and Young's inequality, using the boundedness of the
	cut-off reaction term, we obtain
	\begin{align*}
		\frac{1}{2\tau\sigma}\norm{\ow}_{L^2(\Omega)}^2
		 & +
		\frac{\epsilon}{\sigma}\norm{\nabla\ow}_{L^2(\Omega)}^2
		+
		\frac{\gamma}{\sigma}
		\int_\Omega
		(\phi(\ou)+\phi(\ov))^\gamma |\nabla\ow|^2 \dx x  \leq
		C \sigma,
	\end{align*}
	where $C>0$ depends on $\tau$ but is independent of $\sigma$. Then, since $\sigma\in(0,1)$,
	\begin{align*}
		\norm{\ow}_{H^1(\Omega)}\leq C,
		\qquad \text{and} \qquad \frac{1}{\sigma}\norm{\nabla\ow}_{L^2(\Omega)}^2\leq C,
	\end{align*}
	where $C>0$ is independent of $\sigma$.

	Next, testing the equations for $\ou$ and $\ov$ by $\ou$ and $\ov$,
	respectively, and using Young's inequality, we get
	\begin{align*}
		\frac{1}{2\tau\sigma}\norm{\ou}_{L^2(\Omega)}^2
		+
		\frac{\epsilon}{2\sigma}\norm{\nabla\ou}_{L^2(\Omega)}^2
		 & \leq
		\frac{C}{\sigma}\norm{\nabla\ow}_{L^2(\Omega)}^2
		+
		C\sigma, \\
		\frac{1}{2\tau\sigma}\norm{\ov}_{L^2(\Omega)}^2
		+
		\frac{\epsilon}{2\sigma}\norm{\nabla\ov}_{L^2(\Omega)}^2
		 & \leq
		\frac{C}{\sigma}\norm{\nabla\ow}_{L^2(\Omega)}^2
		+
		C\sigma,
	\end{align*}
	where $C>0$ depends on $\tau$ but is independent of $\sigma$. Combining the previous
	estimates yields
	\begin{align*}
		\norm{\ou}_{H^1(\Omega)}
		+
		\norm{\ov}_{H^1(\Omega)}
		+
		\norm{\ow}_{H^1(\Omega)}
		\leq C_S,
	\end{align*}
	with a constant $C_S>0$ independent of $\sigma$. Hence, the Sch\"afer set is
	bounded in $(L^2(\Omega))^3$ and the Sch\"afer's fixed-point theorem yields the existence of
	$(\ou,\ov,\ow)\in (H^1(\Omega))^3$ such that
	\[
		\mathcal F(\ou,\ov,\ow)=(\ou,\ov,\ow).
	\]
	Moreover, by the conclusion of Step 1, we have
	\[
		\ow=\ou+\ov.
	\]
	The fixed point satisfies, in $H^1(\Omega)'$,
	\begin{subequations}
		\label{eq:scheme-weakForm}
		\begin{align}
			\frac{\ow-w_{k-1}}{\tau}
			- & \epsilon\Delta\ow
			-\gamma\nabla\cdot\left(
			(\phi(\ou)+\phi(\ov))^\gamma\nabla\ow
			\right)
			=
			(R_{u}+R_v )(\phi(\ou),\phi(\ov), \phi(\ow)),\label{eq:scheme-weakForm-w} \\
			\frac{\ou-u_{k-1}}{\tau}
			- & \epsilon\Delta\ou
			-\gamma\nabla\cdot\left(
			\phi(\ou)(\phi(\ou)+\phi(\ov))^{\gamma-1}\nabla\ow
			\right)
			=
			R_{u}(\phi(\ou),\phi(\ov), \phi(\ow)),\label{eq:scheme-weakForm-u}        \\
			\frac{\ov-v_{k-1}}{\tau}
			- & \epsilon\Delta\ov
			-\gamma\nabla\cdot\left(
			\phi(\ov)(\phi(\ou)+\phi(\ov))^{\gamma-1}\nabla\ow
			\right)
			=
			R_v (\phi(\ou),\phi(\ov), \phi(\ow)). \label{eq:scheme-weakForm-v}
		\end{align}
	\end{subequations}
	The homogeneous Neumann boundary conditions are encoded in this variational
	formulation.

	\textbf{Step 3. -- Removal of cut-off}\\
	Let us first note that
	$$
		\phi(\ou)(\ou)_- = 0,
		\qquad
		\phi(\ou)\nabla(\ou)_- = 0,
		\qquad \text{and} \qquad
		\phi(\ov)G_u(\phi(\ow))\geq 0,
	$$
	for almost every $x\in \Omega$. Therefore, upon choosing $\varphi = -(\ou)_-$ as test function, we get
	\begin{align*}
		\frac{1}{\tau}\int_\Omega \abs{(\ou)_-}^2 + \epsilon \abs{\nabla(\ou)_-}^2 \dx x
		 & = - \frac{1}{\tau}\int_\Omega u_{k-1}(\ou)_- \dx x - \int_\Omega \phi(\ov)G_u(\phi(\ow))(\ou)_- \dx x \\
		 & \leq 0.
	\end{align*}
	Hence, $\ou\geq 0$ and, in the same way,  $\ov \geq 0$.
	Next, we show that $\ow\in L^\infty(\Omega)$ using a simple  induction argument. Previously, we chose $M$ such that $w_0\leq M$. Let $0 \leq w_{k-1}\leq M$. Then, we use $\varphi = (\ow-M)_+$ as test function in \eqref{eq:scheme-weakForm-w} to get
	\begin{align*}
		\frac{1}{\tau}\|(\ow-M)_+\|^2_{L^2(\Omega)}
		  & + \int_\Omega
		\left(\epsilon + \gamma(\phi(\ou)+\phi(\ov))^{\gamma}\right)
		\abs{\nabla(\ow-M)_+}^2 \dx x                             \\
		= & \frac{1}{\tau} \int_\Omega (w_{k-1}-M)(\ow-M)_+ \dx x \\
		  & + \int_\Omega
		\left[
			\phi(\ou) F(\phi(\ow))
			+
			\phi(\ov) G(\phi(\ow))
			\right]
		(\ow-M)_+ \dx x .
	\end{align*}
	By assumption, we have
	\begin{align*}
		\int_\Omega \prt*{\phi(\ou) F(\phi(\ow))+\phi(\ov) G(\phi(\ow))}  (\ow-M)_+ \dx x \leq 0.
	\end{align*}
	Moreover, the fact that $w_{k-1}\leq M$ in conjunction with the definition of the positive part yields
	\begin{align*}
		\frac{1}{\tau} \int_\Omega (w_{k-1}-M)(\ow-M)_+ \dx x \leq 0.
	\end{align*}
	Since the left-hand side is nonnegative while the right-hand side is
	nonpositive, we obtain $(\ow-M)_+ = 0$, almost everywhere, and therefore
	$\ow\leq M$. Since we already established $\ow=\ou+\ov$ and
	$\ou,\ov\geq0$, it immediately follows that $\ou,\ov\leq M$, and, as a consequence,
	\[
		\phi(\ou)=\ou,\qquad
		\phi(\ov)=\ov,\qquad
		\phi(\ow)=\ow
		\quad\text{a.e. in }\Omega.
	\]
	Thus, the cut-off function can be removed from \eqref{eq:LinScheme} and
	\eqref{eq:scheme-weakForm}.
\end{proof}

\section{A Priori Estimates}
\label{sec:apriori}
This section is dedicated to deriving estimates for the family of approximations $(u_{k,\epsilon}, v_{k,\epsilon}, w_{k,\epsilon})$ of Scheme \eqref{eq:scheme} that are uniform in both $\epsilon>0$ and $\tau>0$.
We set
\begin{align}
	\label{def:MR}
	M^R \coloneqq \max \left\{
	\sup_{0\leq r \leq M} | F_u(r)|,\sup_{0\leq r\leq M} | F_v(r)|,\sup_{0\leq r \leq M} | G_u(r)|,\sup_{0\leq r\leq M} | G_v(r)|
	\right\}.
\end{align}
Let us recall that $\overline f_{\tau}$ and $\widehat f_{\tau}$ are the piecewise constant interpolation and the piecewise linear interpolation of a sequence, respectively, $(f_k)_{k=0}^N$.

\begin{proposition}[Entropy-Dissipation Estimate]
	\label{prop:entropy-dissipation}
	Let $(u_{k,\epsilon}, v_{k,\epsilon}, w_{k,\epsilon})$ be a solution to Scheme \eqref{eq:scheme}. Then, there holds
	\begin{align*}
		4 \epsilon \int_0^T \int_\Omega \left( \abs*{\nabla \sqrt{\overline    u_{\tau, \epsilon}}}^2 + \abs*{\nabla \sqrt{\overline v_{\tau, \epsilon}}}^2 \right)\dx x \dx t
		+ \frac{4\gamma}{(\gamma + 1)^2} \int_0^T \int_\Omega  \abs*{\nabla \overline w_{\tau, \epsilon}^{(\gamma +1)/2}}^2 \dx x \dx t \leq C
	\end{align*}
	for some $C>0$, independent of $\epsilon, \tau >0$.
\end{proposition}
\begin{proof}
	In the proof, we omit the subscript $\epsilon$ from $(u_{k,\epsilon}, v_{k,\epsilon}, w_{k,\epsilon})$ and simply write $(u_k, v_k, w_k)$ for ease of notation.
	We test the scheme, Eq. \eqref{eq:scheme}, by $\ln (u_k + \delta)$, $\ln (v_k + \delta)$, and $\ln(w_k + \delta)$ with $\delta >0$. Beginning with the equation for $u$, we obtain
	\begin{subequations}
		\label{eq:scheme-tested-log}
		\begin{align}
			\frac1\tau \int_\Omega (u_k - u_{k-1}) \ln(u_k + \delta)\dx x  = & - \epsilon \int_\Omega \frac{\abs*{\nabla u_{k}}^2}{u_k + \delta} \dx x - \int_\Omega \frac{u_k}{u_k + \delta} \nabla u_k \cdot \nabla w_k^\gamma \dx x \\
			\nonumber
			                                                                 & +\int_\Omega R_u(u_k,v_k,w_k) \ln(u_k+\delta) \dx x,
		\end{align}
		and, similarly,
		\begin{align}\label{eq:scheme-tested-log:v}
			\frac1\tau \int_\Omega(v_k - v_{k-1}) \ln(v_k + \delta)\dx x  = & - \epsilon \int_\Omega \frac{\abs*{\nabla v_{k}}^2}{v_k + \delta} \dx x - \int_\Omega \frac{v_k}{v_k + \delta} \nabla v_k \cdot \nabla w_k^\gamma \dx x \\
			\nonumber
			                                                                & +\int_\Omega R_v(u_k,v_k,w_k) \ln(v_k+\delta) \dx x ,
		\end{align}
		for the $v$-equation.
	\end{subequations}
	Then, for $\delta <1$, we compute
	\begin{align*}
		\int_\Omega R_u(u_k,v_k,w_k) \ln(u_k+\delta) \dx x & = \int_\Omega (u_kF_u(w_k)+v_kG_u(w_k))\ln(u_k+\delta) \dx x                 \\
		                                                   & \leq \int_{u_k+\delta \geq 1} (u_kF_u(w_k)+v_kG_u(w_k))\ln(u_k+\delta) \dx x \\
		                                                   & \leq M^R \int_\Omega(u_k+v_k)(u_k+\delta) \dx x\leq C_R,
	\end{align*}
	for some $C_R>0$, independent of $\epsilon, \tau>0$. Here, we used Assumption~\eqref{eq:RuRv_assumptions}. In the same way we can estimate the term in \eqref{eq:scheme-tested-log:v}.  For convenience, we introduce the following notation for the entropy
	\begin{align*}
		\mathcal H[f] := \int_\Omega f \ln f \dx x,
	\end{align*}
	and, using its convexity, Eqs. \eqref{eq:scheme-tested-log}
	\begin{align*}
		\frac{1}{\tau} & \prt*{ \mathcal H[u_k + \delta]  + \mathcal H[v_k + \delta] - \mathcal H[u_{k-1} + \delta] - \mathcal H[v_{k-1} + \delta] } \\
		               & \leq
		- \epsilon \int_\Omega \left( \frac{\abs*{\nabla u_{k}}^2}{u_k + \delta} + \frac{\abs*{\nabla v_{k}}^2}{v_k + \delta}\right) \dx x
		- \int_\Omega \frac{u_k}{u_k + \delta} \nabla u_k \cdot \nabla w_k^\gamma \dx x                                                              \\
		               & \quad - \int_\Omega \frac{v_k}{v_k + \delta} \nabla v_k \cdot \nabla w_k^\gamma \dx x+2C_R.
	\end{align*}
	Using the dominated convergence theorem, we obtain as $\delta \to 0$, that
	\begin{align*}
		\frac{1}{\tau} & \prt{ \mathcal H[u_k]  + \mathcal H[v_k] - \mathcal H[u_{k-1}] - \mathcal H[v_{k-1}]} \\
		               & \leq
		- 4 \epsilon \int_\Omega \left( \abs*{\nabla \sqrt{u_{k}}}^2 + \abs*{\nabla \sqrt{v_{k}}}^2\right)  \dx x
		- \int_\Omega \nabla u_k \cdot \nabla w_k^\gamma \dx x - \int_\Omega  \nabla v_k \cdot \nabla w_k^\gamma \dx x
		+2C_R                                                                                                  \\
		               & =
		- 4 \epsilon \int_\Omega \left( \abs*{\nabla \sqrt{u_{k}}}^2 + \abs*{\nabla \sqrt{v_{k}}}^2\right)  \dx x
		- \int_\Omega \nabla w_k \cdot \nabla w_k^\gamma \dx x+2C_R                                            \\
		               & =
		- 4 \epsilon \int_\Omega \left( \abs*{\nabla \sqrt{u_{k}}}^2 + \abs*{\nabla \sqrt{v_{k}}}^2\right)  \dx x
		- \frac{4\gamma}{(\gamma + 1)^2}\int_\Omega  \abs*{\nabla w_k^{(\gamma +1)/2}}^2 \dx x+2C_R.
	\end{align*}
	In summary, we have
	\begin{align*}
		4 \epsilon \tau \sum_{k=1}^N \int_\Omega \left( \abs*{\nabla \sqrt{u_{k}}}^2 + \abs*{\nabla \sqrt{v_{k}}}^2\right)  \dx x
		+ \frac{4\gamma}{(\gamma + 1)^2} \tau \sum_{k=1}^N \int_\Omega  \abs*{\nabla w_k^{(\gamma +1)/2}}^2 \dx x \leq C,
	\end{align*}
	where
	\begin{align*}
		C := \mathcal H[u_{0}] + \mathcal H[v_{0}] - \mathcal H[u_N]  - \mathcal H[v_N]+2C_RT,
	\end{align*}
	and the statement follows from the definition of the piecewise constant interpolation, Definition \ref{def:pw-interpolations}.
\end{proof}

\begin{lemma}[Space Regularity of Nonlinearity]
	\label{lem:space-regularity-nonlinearity}
	There holds
	\begin{align}
		\norm*{ \overline w_{\tau, \epsilon}^{p}}_{L^2(0,T;H^{1}(\Omega))} \leq C,
	\end{align}
	for arbitrary $p\ge (\gamma +1)/2$, where $C>0$ depends on $p$ but is independent of $\epsilon, \tau >0$.
\end{lemma}
\begin{proof}
	We consider
	\begin{align*}
		\abs*{\nabla \overline w_{\tau,\epsilon}^{p}}^2
		 & = p^2 \abs{\overline w_{\tau,\epsilon}^{p-1} \nabla \overline w_{\tau,\epsilon}}^2                                                            \\
		 & \leq p^2 \abs{\overline w_{\tau,\epsilon}^{p-1-(\gamma-1)/2} \overline w_{\tau,\epsilon}^{(\gamma-1)/2} \nabla \overline w_{\tau,\epsilon}}^2 \\
		 & \leq \frac{4p^2}{(\gamma +1)^2} M^{2p-\gamma-1} \abs{ \nabla \overline w_{\tau,\epsilon}^{(\gamma+1)/2}}^2,
	\end{align*}
	having used the uniform $L^\infty$-bounds from Lemma \ref{lem:buffer-this-didnt-have-a-label} and $2p-\gamma -1\ge 0$. Integrating in space and time yields
	\begin{align*}
		\int_0^T\int_\Omega \abs*{\nabla \overline w_{\tau,\epsilon}^{p}}^2 \dx x \dx t \leq \frac{4p^2}{(\gamma +1)^2} M^{2p-\gamma-1} \int_0^T\int_\Omega \abs{  \nabla \overline w_{\tau,\epsilon}^{(\gamma +1)/2}}^2 \dx x \dx t.
	\end{align*}
	This gives the desired bound as the right-hand side is bounded by Proposition \ref{prop:entropy-dissipation}.
\end{proof}

\begin{lemma}
	\label{lem:regularity-gradu-gradv}
	There exists a constant, $C>0$, independent of $\epsilon, \tau>0$ such that
	\begin{align*}
		\epsilon\norm{\nabla \overline u_{\tau,\epsilon}}_{L^2(\Omega_T)} + \epsilon\norm{\nabla \overline v_{\tau,\epsilon}}_{L^2(\Omega_T)}
		+ \sqrt{\epsilon} \norm{\nabla \overline w_{\tau,\epsilon}}_{L^2(\Omega_T)}\leq C.
	\end{align*}
\end{lemma}
\begin{proof}
	We test Eq. \eqref{eq:scheme:u}  by $\epsilon u_k$, using Young's inequality, and the uniform bounds $M, M^R$ for the density and the growth rates, respectively, and obtain
	\begin{align}
		\frac\epsilon{2\tau} \prt*{\norm{u_k}_{L^2(\Omega)}^2 - \norm{u_{k-1}}_{L^2(\Omega)}^2}+ \frac{\epsilon^2}{2} \norm{\nabla u_k}_{L^2(\Omega)}^2 \leq  C \norm{\nabla w_k^\gamma}_{L^2(\Omega)}^2 + C.
	\end{align}
	Multiplying by $\tau$ and summing over $k$, we get
	\begin{align}
		\epsilon^2 \norm{\nabla \overline u_{\tau,\epsilon}}_{L^2(\Omega_T)}^2  \leq  C\norm{\nabla \overline  w_{\tau,\epsilon}^\gamma}_{L^2(\Omega_T)}^2 + \epsilon (C+ \norm{u_0}_{L^2(\Omega)}^2).
	\end{align}
	Applying the same procedure for the $v_k$-equation yields the corresponding bound, $\epsilon\norm{\nabla \overline v_{\tau,\epsilon}}_{L^2(0,T;L^2(\Omega))}$.

	Test  Eq. \eqref{eq:scheme:w}  by $w_k$ and use the same strategy to obtain
	\begin{align}
		\epsilon \norm{\nabla \overline w_{\tau,\epsilon}}_{L^2(\Omega_T)}^2  + \frac{4\gamma}{(\gamma +2)^2}\norm{\nabla \overline  w_{\tau,\epsilon}^{\gamma/2+1}}_{L^2(\Omega_T)}^2
		\leq
		C+ \frac 12\norm{w_0}_{L^2(\Omega)}^2.
	\end{align}
	Thus, we complete the proof.
\end{proof}

\begin{proposition}[$H^1(0,T;L^2(\Omega))$ and $L^\infty(0,T;H^1(\Omega))$ Estimates]
	\label{prop:timeDeriv}
	Let $(u_k, v_k, w_k)$ be the solution to the scheme. Then, there holds
	\begin{align*}
		\norm{\partial_t \widehat{w_{\tau, \epsilon}^{(\gamma + 2)/2}}}_{L^2(\Omega_T)}
		+
		\norm{\nabla {\overline w_{\tau, \epsilon}^{\gamma + 1}}}_{L^\infty(0,T;L^2(\Omega))}
		\leq C.
	\end{align*}
	and
	\begin{align*}
		\sqrt{\epsilon} \norm{\partial_t \widehat{w_{\tau, \epsilon}}}_{L^2(\Omega_T)}
		+\epsilon\norm{\nabla {\overline w_{\tau, \epsilon}}}_{L^\infty(0,T;L^2(\Omega))}
		 & + \sqrt{\epsilon}
		\norm{\nabla {\overline w_{\tau, \epsilon}^{(\gamma + 2)/2}}}_{L^\infty(0,T;L^2(\Omega))}
		\leq C,
	\end{align*}
	for some constant $C>0$ independent of $\tau, \epsilon >0$.
\end{proposition}
\begin{proof}
	We define
	\begin{align*}
		\Psi(w_k) := \epsilon\dfrac{\gamma+1}{\gamma}  w_k +w_k^{\gamma + 1},
	\end{align*}
	and test Eq.~\eqref{eq:scheme:w} by  $ \Psi(w_k) -  \Psi(w_{k-1}) $ to derive
	\begin{align}
		\label{eq:scheme-tested-psi}
		\mathcal J_k^1 + \mathcal J_k^2  =\mathrm{React}_k,
	\end{align}
	where we set
	\begin{align*}
		\mathcal J_k^1 := \int_\Omega \frac{w_k-w_{k-1}}{\tau} ( \Psi(w_k) -  \Psi(w_{k-1}) ) \dx x,
	\end{align*}
	and
	\begin{align*}
		\mathcal J_k^2 := \frac{\gamma}{\gamma+1} \int_\Omega (\nabla \Psi(w_k) - \nabla  \Psi(w_{k-1}) )\cdot \nabla  \Psi(w_k) \dx x ,
	\end{align*}
	as well as
	\begin{align*}
		\mathrm{React}_k :=\int_\Omega (u_k F(w_k)+v_kG(w_k)) ( \Psi(w_k) -  \Psi(w_{k-1}) ) \dx x.
	\end{align*}
	Let us investigate the first term of the left-hand side by observing
	\begin{align*}
		\mathcal J_k^1 = \int_\Omega & \frac{w_k-w_{k-1}}{\tau} ( \Psi(w_k) -  \Psi(w_{k-1}) ) \dx x                                                                                                                                                      \\
		                             & = \, \frac{\epsilon(\gamma+1)}{\tau \gamma}\int_\Omega (w_k-w_{k-1})( w_k -  w_{k-1} ) \dx x + \frac{1}{\tau}\int_\Omega (w_k-w_{k-1})( w_k^{\gamma+1} -  w_{k-1}^{\gamma+1}  ) \dx x                              \\
		                             & \geq \frac{\epsilon(\gamma+1)}{\tau \gamma}\int_\Omega \abs{ w_k -  w_{k-1} }^2 \dx x + \frac{4(\gamma+1)}{\tau(\gamma+2)^2}\int_\Omega\abs*{ w_k^{\frac{\gamma+2}{2}} -  w_{k-1}^{\frac{\gamma+2}{2}}  }^2 \dx x.
	\end{align*}

	Next, we consider the contribution coming from the right-hand side, and we estimate
	\begin{align*}
		\mathrm{React}_k
		 & = \int_\Omega (u_k F(w_k)+v_kG(w_k)) ( \Psi(w_k) -  \Psi(w_{k-1}) ) \dx x                                                                                         \\
		 & =  \int_\Omega (R_u+R_v) \prt*{ \epsilon\dfrac{\gamma+1}{\gamma}  (w_k -w_{k-1})+(w_k^{\gamma + 1}-w_{k-1}^{\gamma + 1})} \dx x                                   \\
		 & \leq \frac{\epsilon(\gamma+1)}{2\tau \gamma}\int_\Omega \abs{ w_k -  w_{k-1} }^2 \dx x +\tau \frac{\epsilon(\gamma+1)}{2 \gamma}\int_\Omega \abs{R_u+R_v}^2 \dx x \\
		 & \qquad + \int_\Omega  (R_u+R_v) (w_k^{\gamma + 1}-w_{k-1}^{\gamma + 1})\dx x.
	\end{align*}
	Since the function $K(x):= x^{\frac{2(\gamma+1)}{\gamma+2}}$ is convex there holds $ K(x) - K(y) \leq K'(x)(x-y)$, and we can write, for $x=w_k^{\frac{\gamma+2}{2}}$ and $y=w_{k-1}^{\frac{\gamma+2}{2}}$, that
	\begin{align}
		\label{ineq:wtogammap1-wtogammap2o2}
		(w_k^{\gamma + 1}-w_{k-1}^{\gamma + 1}) \leq \frac{2(\gamma+1)}{\gamma+2} w_k^{\frac\gamma2}  \prt*{w_k^{\frac{\gamma+2}{2}}-w_{k-1}^{\frac{\gamma+2}{2}}} .
	\end{align}
	Then, upon multiplying by $R_u + R_v \geq 0$ and using Young's inequality, we derive
	\begin{align*}
		\int_\Omega & (R_u+R_v) (w_k^{\gamma + 1}-w_{k-1}^{\gamma + 1})\dx x                                                                                                                                                                \\
		            & \leq \frac{2(\gamma+1)}{\gamma+2} \int_\Omega \left( \frac{(\gamma+2)\tau }{4}  (R_u+R_v)^2 w_k^\gamma + \frac{1}{(\gamma+2)\tau }\abs*{ w_k^{\frac{\gamma+2}{2}} -  w_{k-1}^{\frac{\gamma+2}{2}}}^2   \right) \dx x.
	\end{align*}
	Thus, we obtain the bound
	\begin{align}
		\mathrm{React}_k \leq \frac{\epsilon (\gamma + 1)}{2\tau\gamma} \int_\Omega \abs*{w_k - w_{k-1}}^2 \dx x +  \frac{2 (\gamma + 1)}{\tau (\gamma + 2)^2} \int_\Omega \abs*{ w_k^{\frac{\gamma+2}{2}} -  w_{k-1}^{\frac{\gamma+2}{2}}}^2 \dx x + C\tau .
	\end{align}
	Thus, we conclude that
	\begin{align}
		\mathcal J_k^1 - \mathrm{React}_k \geq  \frac{\epsilon (\gamma + 1)}{2\tau\gamma} \int_\Omega \abs*{w_k - w_{k-1}}^2 \dx x +  \frac{2 (\gamma + 1)}{\tau (\gamma + 2)^2} \int_\Omega \abs*{ w_k^{\frac{\gamma+2}{2}} -  w_{k-1}^{\frac{\gamma+2}{2}}}^2 \dx x - C\tau.
	\end{align}
	Next, let us address the second term on the left-hand side of \eqref{eq:scheme-tested-psi}, $\mathcal J_k^2$. By using $(a-b)a = \dfrac{1}{2}a^2- \dfrac{1}{2}b^2 +\dfrac{1}{2} (a-b)^2$, it can be rewritten as follows:
	\begin{align*}
		\mathcal J_k^2 = \frac{\gamma}{\gamma+1} & \int_\Omega (\nabla \Psi(w_k) - \nabla  \Psi(w_{k-1}) )\cdot \nabla  \Psi(w_k)\dx x                 \\
		=                                        & \frac{\gamma+1}{\gamma}\epsilon^2 \int_\Omega (\nabla w_k - \nabla  w_{k-1} )\cdot \nabla  w_k\dx x
		+ \epsilon \int_\Omega (\nabla w_k^{\gamma +1} - \nabla  w_{k-1}^{\gamma +1} )\cdot \nabla  w_k\dx x                                           \\
		                                         & +
		\epsilon \int_\Omega (\nabla w_k - \nabla  w_{k-1} )\cdot \nabla  w_k^{\gamma +1}\dx x
		+ \frac{\gamma}{\gamma+1} \int_\Omega (\nabla w_k^{\gamma +1} - \nabla  w_{k-1}^{\gamma +1} )\cdot \nabla  w_k^{\gamma +1}\dx x                \\
		=                                        &
		\frac{\gamma+1}{2\gamma}\epsilon^2
		\left(
		\norm{\nabla w_k}_{L^2(\Omega)}^2
		-
		\norm{\nabla w_{k-1}}_{L^2(\Omega)}^2
		+ \norm{\nabla w_k - \nabla w_{k-1}}_{L^2(\Omega)}^2
		\right)
		+I_k+I\!\!I_k                                                                                                                                  \\
		                                         & +
		\frac{\gamma}{2(\gamma+1)}
		\left(
		\norm{\nabla w_k^{\gamma +1}}_{L^2(\Omega)}^2
		-
		\norm{\nabla w_{k-1}^{\gamma +1}}_{L^2(\Omega)}^2
		+ \norm{\nabla w_k^{\gamma +1} - \nabla w_{k-1}^{\gamma +1}}_{L^2(\Omega)}^2
		\right),
	\end{align*}
	where we have defined
	\[
		I_k := \epsilon \int_\Omega (\nabla w_k^{\gamma +1} - \nabla  w_{k-1}^{\gamma +1} )\cdot \nabla  w_k\dx x ,
		\qquad
		I\!\!I_k := \epsilon \int_\Omega (\nabla w_k - \nabla  w_{k-1} )\cdot \nabla  w_k^{\gamma +1}\dx x.
	\]
	Summing $I_k$ over $k=1,2,\dots, m$, for  $1\leq m\leq N$ and using summation by parts, we obtain
	\begin{align*}
		\sum_{k=1}^m I_k
		= & \epsilon \int_\Omega \nabla w_m^{\gamma +1} \cdot \nabla  w_m\dx x
		-\epsilon \int_\Omega \nabla w_0^{\gamma +1} \cdot \nabla  w_0\dx x                                                          \\
		  & - \epsilon \sum_{k=1}^m \int_\Omega \nabla  w_{k-1}^{\gamma +1} \cdot \left( \nabla  w_k - \nabla  w_{k-1} \right)\dx x.
	\end{align*}
	Therefore, using $ab \le \frac{\alpha}{2}a^2+\frac{1}{2\alpha}b^2$ with $\alpha = \frac{\gamma}{(\gamma+1)\epsilon}$, we have
	\begin{align*}
		\sum_{k=1}^m (I_k+I\!\!I_k)
		=   &
		\frac{4(\gamma+1)}{(\gamma+2)^2}
		\epsilon \norm{\nabla w_m^{(\gamma +2)/2}}_{L^2(\Omega)}^2
		-
		\frac{4(\gamma+1)}{(\gamma+2)^2}
		\epsilon \norm{\nabla w_0^{(\gamma +2)/2}}_{L^2(\Omega)}^2                                                                                                   \\
		    & +
		\epsilon \sum_{k=1}^m\int_\Omega \left( \nabla  w_{k}^{\gamma +1}-\nabla  w_{k-1}^{\gamma +1}\right) \cdot \left( \nabla  w_k - \nabla  w_{k-1} \right)\dx x \\
		\ge &
		\frac{4(\gamma+1)}{(\gamma+2)^2}
		\epsilon \norm{\nabla w_m^{(\gamma +2)/2}}_{L^2(\Omega)}^2
		-
		\frac{4(\gamma+1)}{(\gamma+2)^2}
		\epsilon \norm{\nabla w_0^{(\gamma +2)/2}}_{L^2(\Omega)}^2                                                                                                   \\
		    &
		-
		\frac{\gamma}{2(\gamma+1)}
		\sum_{k=1}^m  \norm{\nabla w_k^{\gamma +1} - \nabla w_{k-1}^{\gamma +1}}_{L^2(\Omega)}^2
		-
		\frac{\gamma+1}{2\gamma}\epsilon^2
		\sum_{k=1}^m  \norm{\nabla w_k - \nabla w_{k-1}}_{L^2(\Omega)}^2 .
	\end{align*}
	Then, upon summing over $k=1, \ldots, m$, we observe that
	\begin{align}
		\sum_{k=1}^m \left(\mathcal J_k^1 + \mathcal J_k^2 - \mathrm{React}_k \right) = 0
	\end{align}
	yields
	\begin{align*}
		\frac{2(\gamma+1)}{(\gamma+2)^2} & \tau \sum_{k=1}^m  \int_\Omega \abs*{ \frac{w_k^{\frac{\gamma+2}{2}} -  w_{k-1}^{\frac{\gamma+2}{2}} }{\tau} }^2 \dx x +\frac{(\gamma+1)}{ 2\gamma} \epsilon\tau \sum_{k=1}^m \int_\Omega \abs*{ \frac{w_k -  w_{k-1}}{\tau} }^2 \dx x \\
		                                 & \qquad
		+ \frac{\gamma+1}{2\gamma}\epsilon^2
		\norm{\nabla w_m}_{L^2(\Omega)}^2
		+\frac{\gamma}{2(\gamma+1)}
		\norm{\nabla w_m^{\gamma+1}}_{L^2(\Omega)}^2
		+\frac{4(\gamma+1)}{(\gamma+2)^2}
		\epsilon \norm{\nabla w_m^{(\gamma +2)/2}}_{L^2(\Omega)}^2                                                                                                                                                                                                                \\
		                                 & \leq \frac{\gamma+1}{2\gamma} \epsilon^2 \norm{\nabla w_0}_{L^2(\Omega)}^2 + \frac{\gamma}{2(\gamma+1)}  \norm{\nabla w_0^{\gamma+1}}_{L^2(\Omega)}^2
		+ \frac{4(\gamma+1)}{(\gamma+2)^2}
		\epsilon \norm{\nabla w_0^{(\gamma +2)/2}}_{L^2(\Omega)}^2
		+C                                                                                                                                                                                                                                                                        \\
		                                 & \leq C,
	\end{align*}
	where $C>0$ is independent of $\epsilon, \tau>0 $, as $w_0\in H^1(\Omega)$ and $w_0\le M$ almost everywhere. This concludes the proof.
\end{proof}

\begin{corollary}[Time and Space Regularities of $w^{\gamma+1}$]
	\label{cor:H1L2_LinfH1_gammap1}
	There holds
	\begin{align*}
		\norm{\partial_t \widehat{w_{\tau,\epsilon}^{\gamma +1}}}_{L^2(\Omega_T)}
		+
		\norm{\nabla \widehat{w_{\tau,\epsilon}^{\gamma +1 }}}_{L^\infty(0,T;L^2(\Omega))}
		\leq C,
	\end{align*}
	where the constant $C>0$ is independent of $\epsilon>0$ and $\tau>0$.
\end{corollary}
\begin{proof}
	By \eqref{ineq:wtogammap1-wtogammap2o2} together with Lemma~\ref{lem:buffer-this-didnt-have-a-label}, we obtain
	\begin{align}
		\norm{\partial_t  \widehat{w_{\tau,\epsilon}^{(\gamma + 1)}}}_{L^2(\Omega_T)} \leq \frac{2(\gamma +1)}{\gamma+2}M^{\gamma/2} \norm{\partial_t  \widehat{w_{\tau,\epsilon}^{(\gamma + 2)/2}}}_{L^2(\Omega_T)},
	\end{align}
	which is bounded by Proposition~\ref{prop:timeDeriv}. The space control follows from the fact that
	\begin{align}
		\norm{\nabla \widehat{w_{\tau,\epsilon}^{\gamma +1 }}}_{L^\infty(0,T;L^2(\Omega))}\leq \norm{\nabla \overline w_{\tau,\epsilon}^{\gamma +1 }}_{L^\infty(0,T;L^2(\Omega))} \leq C
	\end{align}
	from Proposition~\ref{prop:timeDeriv}.
\end{proof}

Our next result pertains to the time regularity of the three densities.
\begin{lemma}[Time Regularity]
	\label{lem:time-regularity}
	There holds
	\begin{align}
		\norm*{\partial_t \widehat u_{\tau, \epsilon}}_{L^2(0,T;H^1(\Omega)')} + \norm*{\partial_t \widehat v_{\tau, \epsilon}}_{L^2(0,T;H^1(\Omega)')} + \norm*{\partial_t \widehat w_{\tau, \epsilon}}_{L^2(0,T;H^1(\Omega)')} \leq C,
	\end{align}
	for some $C>0$, independent of $\epsilon, \tau>0$.
\end{lemma}
\begin{proof}
	We know from the definition of the scheme \eqref{eq:scheme} that
	\begin{align*}
		\norm*{\partial_t \widehat w_{\tau, \epsilon}}_{L^2(0,T;H^1(\Omega)')}
		 & = \ds \sup_{0\neq \phi\in L^2(0,T;H^1(\Omega))} \frac{ \ds \abs*{\int_0^T\skp*{\partial_t \widehat w_{\tau, \epsilon}, \phi}_{H^1(\Omega)', H^1(\Omega)}\dx t}}{\norm{\phi}_{L^2(0,T;H^1(\Omega))}}                   \\
		 & = \sup_{0\neq \phi\in L^2(0,T;H^1(\Omega))} \ds \frac{  \abs*{ \ds \int_0^T\int_\Omega  \nabla \Psi(\overline w_{\tau, \epsilon} ) \cdot  \nabla \phi - (R_u+R_v) \phi \dx x \dx t}}{\|\phi\|_{L^2(0,T;H^1(\Omega))}} \\
		 & \leq C,
	\end{align*}
	having used the uniform $L^2$-bound on $\nabla \Psi$ in Proposition \ref{prop:timeDeriv} and the $L^\infty$-control from Lemma \ref{lem:buffer-this-didnt-have-a-label}. The time regularity of $\partial_t \widehat u_{\tau, \epsilon}$ and $\partial_t \widehat v_{\tau, \epsilon}$ follow a similar argument. Indeed,
	\begin{align*}
		\norm*{\partial_t \widehat u_{\tau, \epsilon}} & _{L^2(0,T;H^1(\Omega)')}
		= \ds \sup_{0\neq \phi\in L^2(0,T;H^1(\Omega))} \frac{ \ds \abs*{\int_0^T\skp*{\partial_t \widehat u_{\tau, \epsilon}, \phi}_{H^1(\Omega)', H^1(\Omega)}\dx t}}{\norm{\phi}_{L^2(0,T;H^1(\Omega))}}                                                                                                                                                                   \\
		                                               & = \sup_{0\neq \phi\in L^2(0,T;H^1(\Omega))} \ds \frac{  \abs*{ \ds \int_0^T\int_\Omega \epsilon \nabla \overline u_{\tau, \epsilon} \cdot \nabla \phi + \overline u_{\tau, \epsilon} \nabla \overline w_{\tau, \epsilon}^{\gamma} \cdot  \nabla \phi - (R_u+R_v) \phi \dx x \dx t}}{\|\phi\|_{L^2(0,T;H^1(\Omega))}} \\
		                                               & \leq \epsilon \norm{\nabla \overline u_{\tau, \epsilon}}_{L^2(\Omega_T)}  +  M \norm{\nabla \overline w_{\tau, \epsilon}^\gamma}_{L^2(\Omega_T)}\| \nabla \Psi( \overline w) \|_{L^2(\Omega_T) }+2M^R\| \overline w\|_{L^2(\Omega_T) }                                                                               \\
		                                               & \leq C
	\end{align*}
	by the uniform estimates on the space regularity provided by Lemma \ref{lem:regularity-gradu-gradv} and Lemma \ref{lem:space-regularity-nonlinearity}. Note that $C>0$ is independent of $\epsilon, \tau>0$.
\end{proof}

\begin{lemma}[Time shifts of the nonlinearity]
	\label{lem:time-shift-nonlinearity}
	Let $\tau >0.$ There holds
	\begin{align*}
		\norm{\overline w_{\tau, \epsilon}^{\gamma+1}(\cdot +\tau)-\overline w_{\tau, \epsilon}^{\gamma+1}}_{L^2(0,T-\tau,L^2(\Omega))}\leq C \tau,
	\end{align*}
	for some constant $C > 0$, independent of $\tau, \epsilon>0$.
\end{lemma}
\begin{proof}
	By virtue of \eqref{ineq:wtogammap1-wtogammap2o2}, we get
	\begin{align*}
		\abs*{\overline w_{\tau, \epsilon}^{\gamma+1}(t+\tau)-\overline w_{\tau, \epsilon}^{\gamma+1}(t)} \leq C(\gamma, M)\abs*{\overline w_{\tau, \epsilon}^{(\gamma+2)/2}(t+\tau)-\overline w_{\tau, \epsilon}^{(\gamma+2)/2}(t)},
	\end{align*}
	where the constant only depends on $M$ and $\gamma$. In particular, it is independent of  $\tau, \epsilon > 0$. Thus, we get
	\begin{align*}
		\norm{\overline w_{\tau, \epsilon}^{\gamma+1}(\cdot + \tau)-\overline w_{\tau, \epsilon}^{\gamma+1}(t)}_{L^2(0,T-\tau,L^2(\Omega))}  \leq C(\gamma, M) \tau,
	\end{align*}
	which concludes the proof.
\end{proof}

\section{Compactness \& Convergence results}
\label{sec:compconv}
We have garnered enough regularity information on our time-discrete approximation to extract convergent subsequences appropriate for the existence result of our time-discrete scheme and its convergence to weak solutions.

\subsection{Existence of solutions to our scheme  ($\epsilon \to 0$)}
This section is dedicated to establishing the existence of solutions to our scheme, that is, we prove Theorem~\ref{thm:existence-scheme}. Let us denote by $(u_{k,\epsilon}, v_{k,\epsilon}, w_{k,\epsilon})_{k=0, \ldots, N; \epsilon>0}$ a family of solutions. By the uniform $L^\infty$-bounds, there exist functions $(u_{k,0}, v_{k,0}, w_{k,0})_{k=0,\ldots, N}\subset  L^\infty(\Omega)$ such that, up to subsequences,
\begin{align}
	u_{k,\epsilon}\overset{*}{\rightharpoonup} u_{k,0}, \quad v_{k,\epsilon}\overset{*}{\rightharpoonup} v_{k,0}, \quad \text{and} \quad w_{k,\epsilon} \overset{*}{\rightharpoonup} w_{k,0},
\end{align}
for all $k=1, \ldots, N$,  weakly-* in $L^\infty(\Omega_T)$, as $\epsilon \to 0$. Moreover, from the uniform-in-$\epsilon$ gradient control in Proposition \ref{prop:timeDeriv}, we also have
\begin{align}
	\nabla w_{k,\epsilon}^{\gamma + 1} \rightharpoonup \nabla w_{k,0}^{\gamma +1},
\end{align}
weakly in $L^2(\Omega)$ (having identified the limit). Moreover
\begin{align}
	w_{k,\epsilon} \to w_{k,0},
\end{align}
pointwise almost everywhere and strongly in $L^p(\Omega)$ for any $1\leq p < \infty$. Finally, let us also recall that
\begin{align}
	\epsilon \norm{\nabla w_{k,\epsilon}} \to 0,
\end{align}
as $\epsilon \to 0$, by Lemma \ref{lem:regularity-gradu-gradv}. These convergences are enough to pass $\epsilon \to 0$ in Eq. \eqref{eq:scheme:w}, to get
\begin{align}
	\label{eq:lim-eps-to-0-w}
	\dfrac1\tau(w_{k,0} - w_{k-1,0}) - \dfrac{\gamma}{\gamma +1} \Delta w_{k,0}^{\gamma + 1} & = (R_u+R_v)(u_{k,0},v_{k,0},w_{k,0})
\end{align}
in the weak sense, where $u_{k,0}, v_{k,0}$ are the weak limits from above. It remains to show that these also satisfy the corresponding equations, Eqs. (\ref{eq:scheme:u}, \ref{eq:scheme:v}).

To this end, we prepare the following proposition.

\begin{proposition}
	\label{prop:strong-convergence-gradwgp1-disc}
	Up to a subsequence, there holds
	\begin{align*}
		\nabla w_{k,\epsilon}^{\gamma + 1}  \to \nabla w_{k,0}^{\gamma+ 1},
	\end{align*}
	strongly in $L^2(\Omega)$, for all $k=1,\ldots, N$,  as $\epsilon \to 0$.
\end{proposition}
\begin{proof}
	We compute the difference between Eq. \eqref{eq:scheme:w} and Eq. \eqref{eq:lim-eps-to-0-w}, to get
	\begin{align}
		\label{eq:diff-wke-wk0}
		\prt*{\frac{w_{k,\epsilon} - w_{k-1,\epsilon}}{\tau} -\frac{w_{k,0} - w_{k-1,0}}{\tau}}
		- \frac{\gamma}{\gamma+1} \Delta (w_{k,\epsilon}^{\gamma+1} - w_{k,0}^{\gamma+1})
		 & = R_{k,\epsilon} - R_{k,0} + \epsilon \Delta w_{k,\epsilon}
	\end{align}
	where
	\begin{align*}
		R_{k,\epsilon} & := R_u(u_{k,\epsilon}, v_{k,\epsilon}, w_{k,\epsilon}) + R_v(u_{k,\epsilon}, v_{k,\epsilon}, w_{k,\epsilon}), \\[0.5em]
		R_{k,0}        & := R_u(u_{k,0}, v_{k,0}, w_{k,0}) + R_v(u_{k,0}, v_{k,0}, w_{k,0}),
	\end{align*}
	with $R_u, R_v$ defined in Eq. \eqref{eq:def-RuRv}. Next, upon  testing Eq. \eqref{eq:diff-wke-wk0} by $(w_{k,\epsilon}^{\gamma+1} - w_{k,0}^{\gamma+1})$ and integrating, we obtain
	\begin{align*}
		 & \int_\Omega \prt*{\frac{w_{k,\epsilon} - w_{k-1,\epsilon}}{\tau} -\frac{w_{k,0} - w_{k-1,0}}{\tau}}\prt*{w_{k,\epsilon}^{\gamma + 1} - w_{k,0}^{\gamma + 1}} \dx x + \frac{\gamma}{\gamma+1} \int_\Omega |\nabla (w_{k,\epsilon}^{\gamma+1} - w_{k,0}^{\gamma+1})|^2\dx x \\
		 & \leq \epsilon \int_\Omega \nabla w_{k,\epsilon} \cdot \nabla(w_{k,\epsilon}^{\gamma+1} - w_{k,0}^{\gamma+1}) \dx x + C_\text{RHS}\norm{w_{k,\epsilon}^{\gamma + 1} - w_{k,0}^{\gamma+1}}_{L^2(\Omega)}                                                                    \\
		 & \leq \frac12 \frac{\gamma}{\gamma+1} \int_\Omega |\nabla (w_{k,\epsilon}^{\gamma+1} - w_{k,0}^{\gamma+1})|^2
		\dx x + C\epsilon^2 \|\nabla w_{k,\epsilon}\|_{L^2(\Omega)}^2 + C_\text{RHS}\norm{w_{k,\epsilon}^{\gamma + 1} - w_{k,0}^{\gamma+1}}_{L^2(\Omega)}.
	\end{align*}
	Thus, we have
	\begin{align*}
		\frac12
		 & \frac{\gamma}{\gamma+1} \int_\Omega |\nabla(w_{k,\epsilon}^{\gamma+1} - w_{k,0}^{\gamma+1})|^2\dx x                                                                                                                                                                     \\
		 & \leq C\epsilon + C_\text{RHS}\norm{w_{k,\epsilon}^{\gamma + 1} - w_{k,0}^{\gamma+1}}_{L^2(\Omega)}  -\int_\Omega \prt*{\frac{w_{k,\epsilon} - w_{k-1,\epsilon}}{\tau} -\frac{w_{k,0} - w_{k-1,0}}{\tau}} \prt*{w_{k,\epsilon}^{\gamma + 1} - w_{k,0}^{\gamma + 1}}\dx x \\
		 & \to 0,
	\end{align*}
	as $\epsilon \to 0$, by the strong convergence of $w_{k,\epsilon}$ in any $L^p$. In conclusion, we have
	\begin{align*}
		\nabla w_{k,\epsilon}^{\gamma + 1}  \to \nabla w_{k,0}^{\gamma+ 1},
	\end{align*}
	strongly in $L^2(\Omega)$, as $\epsilon \to 0$, as claimed.
\end{proof}

In order to pass to the limit in Eq. \eqref{eq:scheme:u} and Eq. \eqref{eq:scheme:v}, the only step missing is the passage to the limit in the term $u\nabla w$ and $v\nabla w$. This can be achieved by observing
\begin{align*}
	\int_\Omega u_{k,\epsilon} \nabla w_{k,\epsilon}^\gamma \cdot \nabla \varphi \dx x = \frac{\gamma}{\gamma + 1}\int_{w_{k,\epsilon>0}} c_{k,\epsilon} \nabla w_{k,\epsilon}^{\gamma+1} \cdot \nabla \varphi \dx x = \frac{\gamma}{\gamma + 1}\int_{\Omega} c_{k,\epsilon} \nabla w_{k,\epsilon}^{\gamma+1} \cdot \nabla \varphi \dx x,
\end{align*}
where $c_{k,\epsilon} := \frac{u_{k,\epsilon}}{w_{k, \epsilon}} \chi_{\{w_{k,\epsilon}>0\}}$. Since $|c_{k,\epsilon}|\leq 1$, it has a weak-* limit in $L^\infty$, denoted by $c_{k,0}$. In conjunction with the strong convergence of $\nabla w_{k,\epsilon}^{\gamma + 1}$ from Proposition \ref{prop:strong-convergence-gradwgp1-disc}, we have
\begin{align*}
	\int_\Omega u_{k,\epsilon} \nabla w_{k,\epsilon}^\gamma \cdot \nabla \varphi \dx x \to  \frac{\gamma}{\gamma + 1}\int_{\Omega} c_{k,0} \nabla w_{k,0}^{\gamma+1} \cdot \nabla \varphi \dx x = \int_{\Omega} c_{k,0} w_{k,0}\nabla w_{k,0}^{\gamma} \cdot \nabla \varphi \dx x.
\end{align*}
The statement follows from identifying $c_{k,0} w_{k,0} = u_{k,0}$. The same procedure applies to Eq. \eqref{eq:scheme:v}, and we identify $(u_{k,0},v_{k,0},w_{k,0})_{k=1,\ldots, N}$ as a solution to our scheme, Scheme~\eqref{eq:intro-scheme}.
Thus, we complete the proof of Theorem~\ref{thm:existence-scheme}.

\subsection{Convergence of our scheme ($\tau \to 0$)}
\label{sec:limit-passage}
Let us summarise.
\begin{remark}
	\label{rem:weak-*-convergence}
	By Lemmas~\ref{lem:buffer-this-didnt-have-a-label}, \ref{lem:space-regularity-nonlinearity}, \ref{lem:time-regularity}, Proposition~\ref{prop:timeDeriv}, Corollary~\ref{cor:H1L2_LinfH1_gammap1}, there exist functions $u,v,w \in L^\infty(\Omega_T)$ and $\widehat u,\widehat v,\widehat w  \in L^\infty(\Omega_T)\cap H^1(0,T;H^1(\Omega)')$ such that $w^{p}, \widehat w^{p}\in L^2(0,T;H^1(\Omega)) \left(\frac{\gamma+1}{2}\le p<+\infty\right)$ and $w^{\gamma+1}, \widehat w^{\gamma+1}\in L^\infty(0,T;H^1(\Omega))$, and up to subsequences,
	\[
		\begin{aligned}
			(\overline{u}_{\tau,0},\overline{v}_{\tau,0},\overline{w}_{\tau,0})
			 & \overset{*}{\rightharpoonup} (u,v,w)
			 &                                                                    & \text{weakly-* in } L^\infty(\Omega_T),                                                      \\
			(\widehat{u}_{\tau,0},\widehat{v}_{\tau,0},\widehat{w}_{\tau,0})
			 & \rightharpoonup (\widehat u,\widehat v,\widehat w)
			 &                                                                    & \text{weakly in } H^1(0,T;H^1(\Omega)')                                                      \\
			 &                                                                    &                                                & \text{and weakly-* in } L^\infty(\Omega_T), \\[0.4em]
			(\overline{w}_{\tau,0}^{\,p},\widehat{w_{\tau,0}^{p}})
			 & \rightharpoonup (w^p, \widehat w^p)
			 &                                                                    & \text{weakly in } L^2(0,T;H^1(\Omega))
			\quad \left(\frac{\gamma+1}{2}\le p<+\infty\right),                                                                                                                  \\[0.4em]
			(\overline{w}_{\tau,0}^{\,\gamma+1},\widehat{w_{\tau,0}^{\gamma+1}})
			 & \overset{*}{\rightharpoonup} (w^{\gamma+1}, \widehat w^{\gamma+1})
			 &                                                                    & \text{weakly-* in } L^\infty(0,T;H^1(\Omega)),
		\end{aligned}
	\]
	as $\tau$ tends to zero.
	Since the piecewise linear and piecewise constant interpolations admit the same weak-* limit, it follows that $\widehat u=u$, $\widehat v=v$ and $\widehat w=w$.
\end{remark}

\begin{lemma}[Relative compactness of the non-linearity]
	\label{lem:rc-nonlinearity}
	Let $w$ be the limit identified in Remark~\ref{rem:weak-*-convergence}.
	Then, up to a subsequence,
	\begin{align*}
		\overline w_{\tau,0}^{\,\gamma+1}
		 & \to w^{\gamma+1}
		 &                  & \text{strongly in } L^q(\Omega_T),                                     \\
		\widehat{w_{\tau,0}^{\,\gamma+1}}
		 & \to w^{\gamma+1}
		 &                  & \text{strongly in } L^q(\Omega_T) \text{ and in } C([0,T];L^q(\Omega))\end{align*}
	for $ 1\leq q<\infty$.
\end{lemma}
\begin{proof}
	From Lemma ~\ref{lem:time-shift-nonlinearity} and Proposition ~\ref{prop:timeDeriv} we already know that
	\begin{align*}
		\frac{1}{\tau} \norm{\overline w_{\tau, 0}^{\gamma+1}(\cdot + \tau)-\overline w_{\tau, 0}^{\gamma+1}(t)}_{L^2(0,T-\tau;L^2(\Omega))}\leq C, \\
		\norm{ \overline w_{\tau, 0}^{\gamma+1}}^2_{L^\infty(0,T;H^1(\Omega))}
		\leq C,
	\end{align*}
	where both constants are independent of $\tau>0$.
	Therefore, we can apply \cite[Thm. 1]{DreherJuengel} to obtain the precompactness of $(\overline w_{\tau, 0}^{\gamma+1})_{\tau}$ in $L^q(0,T;L^2(\Omega))$. Owing to the uniform boundedness, this also implies precompactness in $L^q(\Omega_T)$.
	By considering the difference between
	$\overline{w}_{\tau,0}^{\,\gamma+1}$ and
	$\widehat{w_{\tau,0}^{\,\gamma+1}}$, we can also show the precompactness of
	$\widehat{w_{\tau,0}^{\,\gamma+1}}$ in $L^q(\Omega_T)$.

	In view of Corollary~\ref{cor:H1L2_LinfH1_gammap1}, the Aubin--Lions lemma is applicable, which yields the precompactness of $(\widehat{w_{\tau,\epsilon}^{\gamma +1 }})_\tau$ in $C([0,T];L^2(\Omega))$. Moreover, using the uniform boundedness, one further obtains precompactness in $C([0,T];L^q(\Omega))$.
\end{proof}

\begin{remark}
	\label{rem:strong-convergence}
	By extracting a further subsequence from $(\overline w_{\tau, 0}^{\gamma+1})_{\tau}$, we obtain convergence a.e. Consequently, $\overline w_{\tau, 0}$ converges to $w$ a.e. Owing to the uniform boundedness of $(\overline w_{\tau, 0})_{\tau}$ in
	$L^\infty(\Omega_T)$, the dominated convergence theorem yields the strong convergence of $\overline w_{\tau, 0}$ in $L^q(\Omega_T)$ $(1\le q<\infty)$.
	The same holds for $\widehat{w}_{\tau,0}$.
\end{remark}

\begin{lemma}
	\label{lem:conv-growth}
	For any $\varphi \in L^2(0,T;H^1(\Omega))$ there holds
	\begin{align*}
		\int_0^T\int_\Omega R_u(\overline u_{\tau, 0}, \overline v_{\tau, 0}, \overline w_{\tau, 0}) \varphi\dx x \dx t \to \int_0^T\int_\Omega R_u(u, v, w) \varphi \dx x \dx t,
	\end{align*}
	and
	\begin{align*}
		\int_0^T\int_\Omega R_v(\overline u_{\tau, 0}, \overline v_{\tau, 0}, \overline w_{\tau, 0}) \varphi\dx x \dx t \to \int_0^T\int_\Omega R_v(u, v, w) \varphi \dx x \dx t,
	\end{align*}
	up to a subsequence, as $\tau \to 0$.
\end{lemma}
\begin{proof}
	This is a consequence of the weak-* convergence of $(\overline u_{\tau, 0})_{\tau}, (\overline v_{\tau, 0})_{\tau}$ in $L^\infty(\Omega_T)$ and the strong convergence of $(\overline w_{\tau, 0})_{\tau}$ in any $L^q(\Omega_T)$. Here, we only show this for one term, as all other terms can be treated analogously. We estimate
	\begin{align}
		 & \abs*{\int_0^T\int_\Omega u_{\tau, 0} F_u(\overline w_{\tau, 0}) \varphi\dx x \dx t -\int_0^T\int_\Omega u F_u( w) \varphi\dx x \dx t}                                                \\
		 & \leq \abs*{\int_0^T\int_\Omega u_{\tau, 0} (F_u(w) - F_u(\overline w_{\tau, 0})) \varphi\dx x \dx t}+\abs*{\int_0^T\int_\Omega (\overline u_{\tau, 0} - u) F_u(w) \varphi\dx x \dx t} \\
		 & \to 0,
	\end{align}
	as $\tau \to 0$, by the convergences above.
\end{proof}

Using the interpolations,
Eq.~\eqref{eq:lim-eps-to-0-w} can be rewritten as follows:
\begin{align}
	\label{eq:w-tau-0}
	\partialt {\widehat w_{\tau, 0} } & =  \frac{\gamma}{\gamma + 1}\Delta \overline w_{\tau, 0}^{\gamma + 1}+(R_u+R_v)(\overline u_{\tau, 0},\overline v_{\tau, 0},\overline w_{\tau, 0}).
\end{align}
Taking into account the convergence results obtained above, we can pass to the limit in the weak formulation of Eq.~\eqref{eq:w-tau-0}, and the limit satisfies the following equation in the weak sense.
\begin{align}
	\label{eq:w}
	\partialt  w & =  \frac{\gamma}{\gamma + 1}\Delta  w^{\gamma + 1}+(R_u+R_v)(u, v,w).
\end{align}
In order to pass to the limit in the equations for $u$ and $v$, we establish the following proposition.

\begin{proposition}[Strong Compactness of Gradient]
	The sequences $ (\overline w^{\gamma +1 }_{\tau,0})_{\tau}$ and $ (\widehat{ w^{\gamma +1 }_{\tau,0}})_{\tau}$ converge strongly to $ w^{\gamma +1 } $ in $L^2(0,T;H^1(\Omega))$, as $\tau \to 0$.
\end{proposition}
\begin{proof}
	Let us write $\overline R_{\tau, 0} := (R_u+R_v)(\overline u_{\tau, 0},\overline v_{\tau, 0},\overline w_{\tau, 0})$  and $R := (R_u+R_v)(u,v,w)$. Subtracting Eq.~\eqref{eq:w} from Eq.~\eqref{eq:w-tau-0}, we find
	\begin{align*}
		\fpartial t (\widehat w_{\tau, 0} - w)  - \frac{\gamma}{\gamma + 1}\Delta (\overline w_{\tau, 0}^{\gamma + 1} - w^{\gamma + 1}) = \overline R_{\tau, 0} - R.
	\end{align*}
	We test the equation by $(\overline w_{\tau, 0}^{\gamma +1} - w^{\gamma + 1})$ and get
	\begin{align*}
		\frac{\gamma}{\gamma + 1} & \int_0^T\int_\Omega \abs*{\nabla \overline w_{\tau, 0}^{\gamma +1 } -\nabla w^{\gamma +1 }}^2 \dx x \dx t                                              \\
		=                         & -\int_0^T \skp*{\partial_t \widehat w_{\tau, 0} - \partial_t w, \overline w_{\tau, 0}^{\gamma +1 } - w^{\gamma +1 }}_{H^{-1}(\Omega),H^1(\Omega)}\dx t \\
		                          & +\int_0^T\int_\Omega (\overline R_{\tau, 0} - R)(\overline w_{\tau, 0}^{\gamma +1 } - w^{\gamma +1 })\dx x \dx t                                       \\
		=                         &
		\mathcal I_{\tau, 0}^{1} + \mathcal I_{\tau, 0}^{2} + \mathcal I_{\tau, 0}^{3} + \mathcal I^{4} +  \mathcal I_{\tau, 0}^R,
	\end{align*}
	where
	\begin{align*}
		\left\{
		\begin{array}{rl}
			\mathcal I_{\tau, 0}^{1} \!\!\! & := -\ds\int_0^T\skp*{\partial_t \widehat w_{\tau, 0}, \overline w_{\tau, 0}^{\gamma +1 }}_{H^{1}(\Omega)', H^1(\Omega)} \dx t, \\[1.5em]
			\mathcal I_{\tau, 0}^{2} \!\!\! & := \phantom{-}\ds\int_0^T\skp*{\partial_t \widehat w_{\tau, 0}, w^{\gamma +1 }}_{H^{1}(\Omega)', H^1(\Omega)} \dx t,           \\[1.5em]
			\mathcal I_{\tau, 0}^{3} \!\!\! & := \phantom{-}\ds\int_0^T\skp*{\partial_t  w, \overline w_{\tau, 0}^{\gamma +1 }}_{H^{1}(\Omega)', H^1(\Omega)} \dx t,         \\[1.5em]
			\mathcal I^{4} \!\!\!           & := -\ds \int_0^T\skp*{\partial_t  w,  w^{\gamma +1 }}_{H^{1}(\Omega)', H^1(\Omega)} \dx t,                                     \\[1.5em]
			\mathcal I_{\tau, 0}^R \!\!\!   & := \ds\int_0^T\int_\Omega (\overline R_{\tau, 0} - R)(\overline w_{\tau, 0}^{\gamma +1 } - w^{\gamma +1 })\dx x \dx t.
		\end{array}
		\right.
	\end{align*}
	Let us begin treating each of the integrals, starting with $\mathcal I_{\tau, 0}^1$. We observe that
	\begin{align*}
		\mathcal I_{\tau, 0}^{1} & = -\ds\int_0^T\skp*{\partial_t \widehat w_{\tau, 0}, \overline w_{\tau, 0}^{\gamma +1 }}_{H^{1}(\Omega)', H^1(\Omega)} \dx t         \\[0.5em]
		                         & = -\sum_{k=0}^{N-1}\ds\int_{t_k}^{t_{k+1}}\skp*{\frac{w_{k+1}-w_k}{\tau},  w_{k+1}^{\gamma +1 }}_{H^{1}(\Omega)', H^1(\Omega)} \dx t
	\end{align*}
	Since $f:s\to s^{\gamma+2}/(\gamma +2)$ is convex, we get $f(s)-f(y)\leq f'(s)(s-y)$. By taking $s=w_{k+1}$ and $y=w_k$ we derive
	\begin{align}
		\label{ineq:I1}
		\begin{split}
			\mathcal I_{\tau, 0}^{1}
			 & \leq
			-\frac{1}{\gamma+2} \sum_{k=0}^{N-1}\ds\left( \int_\Omega w_{k+1}^{\gamma +2}\dx x -\int_\Omega w_{k}^{\gamma +2} \dx x \right) \\
			 & =-\frac{1}{\gamma+2} \ds \left( \int_\Omega w_{N}^{\gamma +2}\dx x-\int_\Omega w_{0}^{\gamma +2}\dx x \right).
		\end{split}
	\end{align}
	In view of Lemma~\ref{lem:rc-nonlinearity}, we have the strong convergence of a subsequence of
	$(\widehat{w_{\tau, 0}^{\gamma + 1}})_\tau$ to $w ^{\gamma + 1}$ in $C([0,T],L^q(\Omega))$ ($1\le q<\infty$), which implies
	\[
		\int_\Omega \left( \widehat{w_{\tau, 0}^{\gamma + 1}}(t) \right)^q \dx x
		\to
		\int_\Omega \left( w^{\gamma + 1}(t) \right)^q\dx x,
	\]
	for all $t\in [0,T]$, as $\tau \to 0$. Hence,
	\begin{align*}
		\int_\Omega w_{N}^{\gamma +2}\dx x
		 & =
		\int_\Omega w_{N}^{\gamma +1\frac{\gamma+2}{\gamma+1}}\dx x
		=
		\int_\Omega \left( \widehat{w_{\tau, 0}^{\gamma + 1}}(T) \right)^{\frac{\gamma+2}{\gamma+1}}\dx x \\
		 & \to
		\int_\Omega \left( {w^{\gamma + 1}}(T) \right)^{\frac{\gamma+2}{\gamma+1}}\dx x
		=
		\int_\Omega {w^{\gamma + 2}}(T)\dx x.
	\end{align*}

	This yields that the right-hand side of \eqref{ineq:I1} converges to $\mathcal I^4$.
	Because of weak convergence, it follows that $\mathcal I_{\tau, 0}^2 \to - \mathcal I^4$, and $\mathcal I_{\tau, 0}^3 \to -\mathcal I^4$, as $\tau \to 0$.

	Next, let us turn our attention to the integral $\mathcal I_{\tau, 0}^R$ and prove $\mathcal I_{\tau, 0}^R\to 0$, as $\tau \to 0$. We know from Lemma ~\ref{lem:rc-nonlinearity} that up to a subsequence $\overline w^{\gamma+1}_{\tau, 0}$ converges to $ w^{\gamma+1}_{\tau, 0}$ in $L^q(\Omega_T)$, $q\geq 1$. Therefore, as a byproduct, we get that $\overline w_{\tau, 0}$ converges almost everywhere. Furthermore, we know that $\norm{\overline w_{\tau, 0}}_{L^\infty(\Omega_T)}\leq M$. Therefore, we get by Lebesgue's dominated convergence theorem that $\overline w_{\tau, 0}\to w$, strongly in $L^p(\Omega_T)$, $1\leq p< \infty$, as $\tau \to 0$.

	Since $\overline R_{\tau, 0}$ and $R$ are bounded we end up with
	\begin{align}
		\mathcal I_{\tau, 0}^R & = \ds\int_0^T\int_\Omega (\overline R - R)(\overline w_{\tau, 0}^{\gamma +1 } - w^{\gamma +1 })\dx x \dx t \to 0,  \quad \tau \to 0.
	\end{align}

	Hence,
	\begin{align}
		\lim_{\tau \to 0} \int_0^T\int_\Omega \abs*{\nabla \overline w_{\tau, 0}^{\gamma +1 } -\nabla w^{\gamma +1 }}^2 \dx x \dx t = 0,
	\end{align}
	along a suitable subsequence.
	By considering the difference between
	$\overline{w}_{\tau,0}^{\,\gamma+1}$ and
	$\widehat{w_{\tau,0}^{\,\gamma+1}}$, we can also show the strong convergence of
	$\widehat{w_{\tau,0}^{\,\gamma+1}}$ to $w^{\gamma +1 }$ in $L^2(0,T;H^1(\Omega))$.
\end{proof}

Now, having garnered all necessary compactness results we may now draw our attention to the convergence result. To this end, let us revisit \eqref{eq:intro-scheme}, i.e.,
\begin{align*}
	\begin{split}
		\left\{
		\begin{array}{rll}
			\dfrac1\tau(u_k - u_{k-1}) - \nabla \cdot\prt*{u_k \nabla w_k^\gamma} \!\!\!  & = R_u(u_k,v_k,w_k), \\[1em]
			\dfrac1\tau(v_k - v_{k-1})  - \nabla \cdot\prt*{v_k \nabla w_k^\gamma} \!\!\! & = R_v(u_k,v_k,w_k). \\
		\end{array}
		\right.
	\end{split}
\end{align*}
Thus, for any given $\varphi \in L^2(0,T;H^1(\Omega))$, we have
\begin{align}
	\label{eq:almost-weak-form}
	\begin{split}
		\left\{
		\begin{array}{rll}
			\ds \int_0^T\left\langle \partialt {\widehat u_{\tau, 0}}, \varphi \right \rangle_{H^1(\Omega)', H^1(\Omega)}\dx t
			  & + \ds\int_0^T\int_\Omega \overline u_{\tau, 0}  \nabla \overline w_{\tau, 0}^\gamma \cdot \nabla \varphi \dx x \dx t \!\! \\
			= & \ds  \int_0^T\int_\Omega R_u(\overline u_{\tau, 0}, \overline v_{\tau, 0} , \overline w_{\tau, 0}) \varphi\dx x \dx t,    \\[1.5em]
			\ds \int_0^T\left\langle \partialt {\widehat v_{\tau, 0}}, \varphi \right \rangle_{H^1(\Omega)', H^1(\Omega)}\dx t
			  & + \ds\int_0^T\int_\Omega \overline v_{\tau, 0}  \nabla \overline w_{\tau, 0}^\gamma \cdot \nabla \varphi \dx x \dx t \!\! \\
			= & \ds \int_0^T\int_\Omega R_v(\overline u_{\tau, 0}, \overline v_{\tau, 0} , \overline w_{\tau, 0}) \varphi\dx x \dx t.
		\end{array}
		\right.
	\end{split}
\end{align}
In the following lemma we show the convergence of the velocity term
\begin{lemma}
	\label{lem:conv-velo}
	For any $\varphi \in L^2(0,T;H^1(\Omega))$ there holds
	\begin{align*}
		\int_0^T\int_\Omega \overline u_{\tau, 0} \nabla \overline w_{\tau, 0}^{\gamma} \cdot \nabla \varphi \dx x \dx t \to \int_0^T\int_\Omega u \nabla w^{\gamma} \cdot \nabla \varphi \dx x \dx t,
	\end{align*}
	and
	\begin{align*}
		\int_0^T\int_\Omega \overline v_{\tau, 0} \nabla \overline w_{\tau, 0}^{\gamma} \cdot \nabla \varphi \dx x \dx t \to \int_0^T\int_\Omega v \nabla w^{\gamma} \cdot \nabla \varphi\dx x \dx t,
	\end{align*}
	up to a subsequence, as $\tau \to 0$.
\end{lemma}
\begin{proof}
	Define $\eta_{\tau, 0} :=\frac{\overline u_{\tau, 0}}{\overline w_{\tau, 0}}\chi(\overline w_{\tau, 0})$.
	Since $0 \leq \overline \eta_{\tau, 0} \leq 1$, let us denote $\eta := L^\infty$-$\lim^*_{\tau \to 0} \overline \eta_{\tau, 0}$.
	If $w = L^2$-$\lim_{\tau \to 0} \overline w_{\tau, 0}$ and $u = L^\infty$-$\lim^*_{\tau \to 0 } \overline u_{\tau, 0}$, we see, on the one hand, that
	\begin{align*}
		\int_0^T\int_\Omega \zeta \overline u_{\tau, 0} \dx x \dx t \to \int_0^T\int_\Omega \zeta u \dx x \dx t.
	\end{align*}
	and on the other hand
	\begin{align*}
		\int_0^T\int_\Omega \zeta \overline \eta_{\tau, 0} \overline w_{\tau, 0}  \dx x \dx t \to \int_0^T\int_\Omega \zeta \eta w \dx x \dx t
	\end{align*}
	for all $\zeta \in L^1(\Omega_T)$,
	from where we conclude that $u = \eta w$.
	Now, we are ready to address the velocity term. Indeed, we may estimate
	\begin{align*}
		 & \abs*{\int_0^T\int_\Omega \overline u_{\tau, 0} \nabla \overline w_{\tau, 0}^\gamma \cdot \nabla \varphi \dx x \dx t - \int_0^T\int_\Omega u \nabla  w^\gamma \cdot \nabla \varphi \dx x \dx t }                                                               \\
		 & = \abs*{\int_0^T\int_{\overline w_{\tau, 0} > 0}  \overline u_{\tau, 0} \nabla  \overline w_{\tau, 0}^\gamma \cdot \nabla \varphi \dx x \dx t - \int_0^T\int_{w>0} u \nabla w^\gamma \cdot \nabla \varphi \dx x \dx t }                                        \\
		 & = \abs*{\int_0^T\int_{\overline w_{\tau, 0} > 0}  \overline \eta_{\tau, 0} \overline w_{\tau, 0} \nabla \overline w_{\tau, 0}^\gamma \cdot \nabla \varphi \dx x \dx t - \int_0^T\int_{w>0} \eta w \nabla  w^\gamma \cdot \nabla \varphi \dx x \dx t }          \\
		 & = \frac{\gamma}{\gamma + 1}\abs*{\int_0^T\int_{\overline w_{\tau, 0} > 0}  \overline \eta_{\tau, 0} \nabla \overline w_{\tau, 0}^{\gamma+1} \cdot \nabla \varphi \dx x \dx t - \int_0^T\int_{w>0} \eta \nabla  w^{\gamma+1} \cdot \nabla \varphi \dx x \dx t } \\
		 & = \frac{\gamma}{\gamma + 1}\abs*{\int_0^T\int_{\Omega}  \left(\overline \eta_{\tau, 0} \nabla \overline w_{\tau, 0}^{\gamma+1} - \eta \nabla  w^{\gamma+1}\right) \cdot \nabla \varphi\dx x \dx t }                                                            \\
		 & \leq \frac{\gamma}{\gamma + 1}\abs*{\int_0^T\int_{\Omega}  (\overline \eta_{\tau, 0} -\eta) \nabla  w^{\gamma+1} \cdot \nabla \varphi\dx x \dx t }                                                                                                             \\
		 & \qquad +\frac{\gamma}{\gamma + 1}\abs*{\int_0^T\int_{\Omega}  \overline \eta_{\tau, 0} \left(\nabla \overline w_{\tau, 0}^{\gamma+1} - \nabla  w^{\gamma+1}\right) \cdot \nabla \varphi\dx x \dx t }                                                           \\
		 & \to 0,
	\end{align*}
	by weak * convergence of $\overline \eta_{\tau, 0} \to^* \eta$ and strong convergence of $\nabla \overline w_{\tau, 0}^{\gamma+1} \to \nabla w^{\gamma+1}$.
\end{proof}

\begin{proof}[Proof of Theorem~\ref{thm:convergence}]
	We have now gathered all the ingredients needed to pass to the limit in System \eqref{eq:almost-weak-form}.
	By passing to the limit $\tau\to0$, it follows that $(u,v)$ is a weak solution of System \eqref{eq:cts-system}.

	This concludes the proof of Theorem~\ref{thm:convergence}.
\end{proof}

\bibliographystyle{plain}
\bibliography{references}

\end{document}